\newif\ifarxiv
\arxivtrue
\ifarxiv
    \documentclass[a4paper]{article}
    \usepackage[margin=2.5cm]{geometry}
    \newcommand{\headers}[2]{\pagestyle{myheadings}\markboth{\expandafter{#2}}{\expandafter{#1}}}
    \newcommand\funding[1]{\protect\\ \hspace*{15.37pt}{\bfseries Funding:} #1}
    \newcommand{\email}[1]{\protect\href{mailto:#1}{#1}}
    \usepackage{comment}
    \excludecomment{keywords}
    \excludecomment{MSCcodes}
    \usepackage{amsthm}
    \usepackage{amsmath}
    \usepackage{aliascnt}
    \usepackage{hyperref}
    \usepackage{cleveref}
    \newtheorem{theorem}{Theorem}[section]
    \newaliascnt{lemma}{theorem}
    \newtheorem{lemma}[lemma]{Lemma}
    \aliascntresetthe{lemma}
    \newaliascnt{corollary}{theorem}
    \newtheorem{corollary}[corollary]{Corollary}
    \aliascntresetthe{corollary}
    \newaliascnt{proposition}{theorem}
    \newtheorem{proposition}[proposition]{Proposition}
    \aliascntresetthe{proposition}
    \newaliascnt{definition}{theorem}
    \newtheorem{definition}[definition]{Definition}
    \aliascntresetthe{definition}
    \crefname{theorem}{theorem}{theorems}
    \Crefname{theorem}{Theorem}{Theorems}
    \crefname{lemma}{lemma}{lemmas}
    \Crefname{lemma}{Lemma}{Lemmas}
    \crefname{corollary}{corollary}{corollaries}
    \Crefname{corollary}{Corollary}{Corollaries}
    \crefname{proposition}{proposition}{propositions}
    \Crefname{proposition}{Proposition}{Propositions}
    \crefname{definition}{definition}{definitions}
    \Crefname{definition}{Definition}{Definitions}

\else
  \documentclass[review,hidelinks,onefignum,onetabnum]{siamart250211}
\fi

\ifarxiv
  \newcommand{\maybemath}[1]{\[#1\]}
\else
  \newcommand{\maybemath}[1]{$#1$}
\fi

\usepackage[utf8]{inputenc}
\usepackage{graphicx} 
\usepackage[export]{adjustbox}

\usepackage{amssymb}
\usepackage{amsfonts}
\usepackage{amsmath}

\usepackage{hyperref}
\usepackage{cleveref}
\usepackage[dvipsnames]{xcolor}
\usepackage{ifthen}

\newcommand{\Sm}{\mathcal{S}^m}
\newcommand{\Stwo}{\mathcal{S}^2}

\usepackage{xspace}

\newboolean{final}
\setboolean{final}{true}
\setboolean{final}{false}
\ifthenelse{\boolean{final}}{\newcommand{\todo}[1]{}}
{\newcommand{\todo}[1]{{\color{red}TODO: #1}}}
\ifthenelse{\boolean{final}}{\newcommand{\anne}[1]{}}{\newcommand{\anne}[1]{{\color{magenta}Anne: #1}}}
\ifthenelse{\boolean{final}}{\newcommand{\luka}[1]{}}{\newcommand{\luka}[1]{{\color{Green}Luka: #1}}}
\ifthenelse{\boolean{final}}{\newcommand{\dimo}[1]{{}}}{\newcommand{\dimo}[1]{{\color{cyan}Dimo: #1}}} 
\ifthenelse{\boolean{final}}{\newcommand{\tea}[1]{}}{\newcommand{\tea}[1]{{\color{Orange}Tea: #1}}}
\ifthenelse{\boolean{final}}{\newcommand{\new}[1]{{#1}}}{\newcommand{\new}[1]{{\color{blue} #1}}}
\ifthenelse{\boolean{final}}{}{}
\ifthenelse{\boolean{final}}{\newcommand{\del}[1]{{}}}{\newcommand{\del}[1]{{ \color{Tan} #1 } }}

\newcommand{\PS}{{\rm PS}}
\newcommand{\R}{\mathbb{R}}

\newcommand{\ftheta}{\mathcal{F}_\theta}
\newcommand{\Image}{{\rm Im}}
\newcommand{\COBI}{{\textsf{COBI}}\xspace} 
\newcommand{\hide}[1]{}

\DeclareMathOperator{\Proj}{Proj}
\DeclareMathOperator{\Lev}{Lev}
\DeclareMathOperator{\argmin}{argmin}

\DeclareMathOperator{\ND}{ND}

\renewcommand{\Upsilon}{\Psi}

\ifpdf
  \DeclareGraphicsExtensions{.eps,.pdf,.png,.jpg}
\else
  \DeclareGraphicsExtensions{.eps}
\fi

\headers{Pareto Set Characterization and the COBI Problem Generator}{A. Auger, D. Brockhoff, L. Oprav\v{s}, and T. Tu\v{s}ar}

\newcommand{\mytitle}{Pareto Set Characterization in Constrained Multiobjective Optimization and the \COBI{} Problem Generator}
\newcommand{\myfunding}{We acknowledge financial support from the Slovenian Research and Innovation Agency (research core funding No.\ P2-0209 and projects No.\ N2-0254 ``Constrained Multiobjective Optimization Based on Problem Landscape Analysis'' and GC-0001 ``Artificial Intelligence for Science''). This work was also motivated by the COST Action CA22137 ``Randomised Optimisation Algorithms Research Network'' (ROAR-NET), supported by the European Cooperation in Science and Technology.}
\ifarxiv
\title{\mytitle\thanks{\textbf{Funding:} \myfunding}}
\else
\title{Pareto Set Characterization in Constrained Multiobjective Optimization and the \COBI{} Problem Generator
\thanks{Submitted to the editors on \today~\todo{Check date before submission}.
\funding{\myfunding}}}
\fi

\ifarxiv
\author{
Anne Auger\textsuperscript{1}
\and
Dimo Brockhoff\textsuperscript{1}
\and
Luka Oprav\v{s}\textsuperscript{2}
\and
Tea Tu\v{s}ar\textsuperscript{2}
\\[1ex]
\textsuperscript{1}Inria, CMAP, CNRS, École Polytechnique, Institut Polytechnique de Paris,\\ Palaiseau, France
\texttt{\href{mailto:firstname.lastname@inria.fr}{firstname.lastname@inria.fr}}\\[0.5ex]
\textsuperscript{2}Jo\v{z}ef Stefan Institute, Ljubljana, Slovenia\\
\texttt{\href{mailto:firstname.lastname@ijs.si}{firstname.lastname@ijs.si}}
}
\date{}
\else
\author{
Anne Auger\thanks{Inria, CMAP, CNRS, École Polytechnique, Institut Polytechnique de Paris, Palaiseau, France (\email{firstname.lastname@inria.fr}).}
\and
Dimo Brockhoff\footnotemark[2]
\and
Luka Oprav\v{s}\thanks{Jo\v{z}ef Stefan Institute, Ljubljana, Slovenia (\email{firstname.lastname@ijs.si}).}
\and
Tea Tu\v{s}ar\footnotemark[3]
}
\fi

\usepackage{amsopn}

\ifpdf
\hypersetup{
  pdftitle={Pareto Set Characterization in Constrained Multiobjective Optimization and the \COBI{} Problem Generator},
  pdfauthor={Anne Auger, Dimo Brockhoff, Luka Opravs and Tea Tusar}
}
\fi

\begin{document}

\maketitle

\begin{abstract}
Benchmark problems play a central role in assessing the performance of numerical optimization algorithms. However, many existing constrained multiobjective optimization benchmark problems rely on overly restricted constructions or lack formal analysis of their optimal solution sets, limiting their relevance for systematic algorithm evaluation.
In this work, we introduce a class of analytically tractable constrained multiobjective optimization problems whose Pareto sets can be formally characterized. The construction is based on convex-quadratic functions with positive definite Hessians, combined through multipeak formulations in which each objective is defined as the minimum over several convex-quadratic components. This approach preserves analytical structure while enabling multimodality (non-convexity), ill-conditioning and non-separability. The constraints are built as sublevel sets of multipeak functions giving rise to problems with potentially disconnected feasible regions. 
Building on these results, we propose \COBI{}, a scalable generator of constrained bi-objective test problems designed for benchmarking derivative-free optimization algorithms. 
We provide a reference Python implementation that enables straightforward integration of \COBI{} instances into benchmarking workflows.
\end{abstract}



		\crefname{proposition}{proposition}{propositions}
		\Crefname{proposition}{Proposition}{Propositions}

\section{Introduction}
Test problems are essential to evaluate the convergence of numerical optimization algorithms as well as for benchmarking them~\cite{more2009benchmarking}. In order to ensure that performance on test problems generalizes to real-world applications, these problems should capture the key challenges encountered in those applications---a design rationale that has been followed, for example, for the design of the test functions of the COmparing Continuous Optimizers (COCO) platform~\cite{hansen2021coco}. At the same time, problems should have well-understood properties that allow for systematic analysis, performance quantification and diagnosis of algorithmic weaknesses. In this context, a fundamental prerequisite for analyzing any \emph{multiobjective} numerical optimization (MO) algorithm on a test problem is the knowledge of the optimal solution set, known as the \emph{efficient set} or \emph{Pareto set}, and its image in the objective space, the \emph{Pareto front}. Without access to these sets, it is difficult and often impossible to determine whether an algorithm is converging to Pareto-optimal solutions or to meaningfully assess its performance quantitatively. 
Equally important, however, is that the structure and location of the Pareto set and front are not only known but also theoretically justified. In some cases, test problems have been introduced under the assumption of certain properties (such as a degenerate Pareto front) that were later (empirically) shown to be inaccurate~\cite{ishibuchi2015pareto}. This underlines the need for test problems that not only present meaningful challenges and have known optima, but whose Pareto sets and fronts are grounded in formal theoretical analysis.

Despite the importance of well-designed test problems, the field of constrained multiobjective optimization (CMO) lacks robust benchmarks with both theoretical foundations and practical relevance. While existing test problems often have known optima, they tend to be overly restricted and lack formal analysis. Our primary motivation is to address this gap by introducing CMO test problems with mathematically characterized Pareto sets that reflect key difficulties encountered in real-world scenarios. By a \emph{known} Pareto set, we mean one that is analytically characterized and efficiently computable.

To achieve this, we build on a class of objective functions central to numerical optimization: convex-quadratic functions with positive definite Hessians. These functions can exhibit ill-conditioning and non-separability---two fundamental challenges in both local and global optimization. They have historically served as building blocks in the design of major optimization algorithms, from quasi-Newton methods to the Covariance Matrix Adaptation Evolution Strategy (CMA-ES) family in stochastic optimization.
Multipeak formulations, where each objective is defined as the minimum over several convex-quadratic `peak' functions, have been used to introduce multimodality and non-convexity into otherwise well-structured problems~\cite{toure2019bi,glasmachers2019challenges}. As we will prove, this construction retains analytical tractability while enabling rich objective landscapes with multiple attraction basins.

By pairing these multipeak objective functions with combinations of linear and convex-quadratic constraints (which can result in multi-modal constraints and disconnected feasible sets), we obtain a broad class of constrained multiobjective optimization problems. For this class, we derive an explicit characterization of the Pareto set.
Additionally, by exploiting structural invariance properties of the Pareto set---specifically its invariance under sign-preserving transformations of inequality constraints, zero-only-at-zero transformations of equality constraints and strictly increasing transformations of the objective functions---we extend this characterization to a substantially larger family of problems.

On this basis, we introduce \COBI{}, a scalable test problem generator for derivative-free COnstrained BI-objective optimization. \COBI{} captures real-world challenges such as non-separable, ill-conditioned, multimodal, non-smooth and discontinuous objectives, and disconnected feasible regions. Unlike many existing benchmarks, it is grounded in formal mathematical results that precisely characterize the structure and location of the Pareto set and front and allows for their efficient approximation via conic optimization splitting solvers. We also provide a ready-to-use implementation of the resulting problems to facilitate benchmarking.

In summary, this work makes the following contributions:
(a) A mathematical characterization of Pareto sets in constrained multiobjective problems composed of strictly convex-quadratic functions, extended to multipeak constructions with non-convex feasible regions.
(b) Proof of invariance of the resulting Pareto sets under strictly increasing transformations of the objectives and sign-preserving transformations of the constraints, thereby broadening the class of analytically tractable problems.
(c) A formal definition of the scalable \COBI{} test problem generator, along with methods for numerically approximating its Pareto sets.
(d) A reference Python implementation\footnote{\label{foot:cobi}\url{https://github.com/numbbo/cobi-problem-generator/}} that enables seamless integration of \COBI{} instances into benchmarking workflows.
(e) Visualization and analysis of key problem features, including non-separability, ill-conditioning and multimodality in both objectives and constraints—that reflect challenges in real-world scenarios.

Note that we consider a derivative-free optimization setting, as our primary goal is to construct benchmark problems for evaluating derivative-free solvers. Nevertheless, if the transformations applied to the objectives are omitted or chosen so as to preserve differentiability, the resulting problem instances are compatible with gradient-based optimization methods.

\vspace*{-1em}
\subsection*{Related work}
Most CMO benchmark problems were designed to have a known Pareto front and often also a known Pareto set~\cite{fan2020difficulty,liang2022survey,liu2019handling,ma2019evolutionary,to2017a,zhou2020constrained}. By construction, however, they rarely exhibit difficulties observed in practice, such as ill-conditioning or Pareto sets not aligned with the coordinate axes. In fact, their variables are typically artificially partitioned into distance and position variables\footnote{A \emph{distance variable} never yields non-comparable solutions when varied alone. A \emph{position variable} yields only non-dominated (or equivalent) solutions when varied alone~\cite{huband2005scalable}.}, which introduces biases in algorithm evaluation and limits their representativeness of real-world optimization landscapes~\cite{ishibuchi2023performance}. Moreover, constraints are often defined solely by specifying infeasible regions in the objective space without a clear effect on the Pareto set. 

Another class of CMO benchmark problems with known, though not formally characterized, Pareto sets consists of distance-based problems with constraints \cite{fieldsend2021visualizable}. While designed to be easily visualizable, their Pareto sets exhibit a highly regular structure (in the bi-objective case, they consist of simple line segments), and three of the four types of added constraints do not introduce optimal solutions beyond those of the unconstrained problem. 

More complex problems can be constructed by combining single-objective functions with known properties. However, even without constraints, this typically yields Pareto sets with intricate structures that are not available in closed form or theoretically characterized. A prominent example is the {\ttfamily bbob-biobj} suite of the \href{https://coco-platform.org}{COCO platform}~\cite{brockhoff2022using,hansen2021coco}. The same holds for most real-world-inspired test problems~\cite{kumar2021benchmark,tanabe2020re}.

For the special case of convex objective functions subject to linear or convex constraints, numerous algorithmic approaches exist (see, e.g., \cite{ruzika2005approximation} for an overview of early methods). They typically exploit the fact that minimizing a weighted sum of the objectives yields Pareto-optimal solutions and that every Pareto-optimal solution can be obtained as the optimum of a weighted-sum scalarization for an appropriate choice of weights~\cite{jahn2009a}.

Analytical characterizations of the entire Pareto set of concrete constrained problems, however, are rare. For instance, \cite{qi2025analytical} analyzes the Pareto set of a portfolio optimization problem with one quadratic and several linear objectives under specific constraints. To the best of our knowledge, there is no theoretical analysis of the Pareto sets of constrained problems with multi-modal objectives, in particular when the objectives are defined as minima of strictly convex-quadratic subfunctions.

To address this gap, we build on the multipeak construction paradigm, where each objective is defined as the minimum of several strictly convex-quadratic functions. The combination of basic functions with elliptical sublevel sets as a single-objective test problem was introduced in~\cite{preuss2004importance}, refined with an exponential decrease of the objective function in~\cite{gallagher2006general} and further developed as Multi-Peaks~2 in~\cite{wessing2015multiple}. These constructions were subsequently extended to the multiobjective setting in~\cite{kerschke2019search,schaepermeier2023multipeaks}. All of these multimodal benchmark problems were developed only in the unconstrained setting.

\subsection*{Paper organization}
Section~\ref{sec:mathBG} introduces the mathematical background for CMO, followed by the main theoretical results in Section~\ref{sec:main-theory}, including the characterization of Pareto sets for constrained convex-quadratic problems, their extension to multipeak settings and key invariance properties. Section~\ref{sec:testfunctiongenerator}  presents the \COBI{} problem generator, Sections~\ref{sec:approximations} and \ref{sec:visualinspection} describe Pareto set approximation methods and analyze the resulting \COBI{} problem properties. Section~\ref{sec:perfassessment} demonstrates their use in benchmarking.

\subsection*{Notations}
Given a symmetric matrix $A \in \R^{n\times n}$, we write $A \succeq 0$ when $A$ is positive semi-definite, and $A \succ 0$ when $A$ is positive definite. We denote $\R^m_+ = \{ y \in \R^m, y_i \geq 0\}$ and the unit simplex of $\R^m$ by $\Sm = \{ \theta \in \R^m,\mbox{ with }
\theta_i \geq 0, \sum_i \theta_i =1  \}
$. For $m=2$, the simplex is denoted by $\Stwo$. We denote by $\Image(f)$ the image of a function $f: \R^n \to \R$, i.e., the set of $y \in \R$ such that there exists $x \in \R^n$ such that $f(x) = y$. Given two sets $A$ and $B$, $A \subset B$ means that all elements of $A$ are included in $B$. Thus, $A$ and $B$ can also be equal.

\section{Mathematical background}\label{sec:mathBG}
We present in this section the mathematical background of constrained multiobjective optimization and the different mathematical tools and concepts on which our proofs are based.

\subsection{Problem formulation, optimality and feasibility notations}
We consider an abstract numerical constrained multiobjective optimization problem with $m \in \mathbb{N}$ objectives, $p$ inequality constraints and $q$ equality constraints 
\begin{equation}
\begin{aligned}\label{eq:MO-constrained}
 \mbox{minimize} \quad f(x) &= \left( f_1(x),\ldots,f_m(x) \right)\\
 \mbox{subject to} \quad  g_i(x) &\leq 0, i=1, \ldots, p ; \,\,
 h_j(x) = 0, 
 j=1,\ldots, q
 \end{aligned}
\end{equation}
with numerical objectives $f_i: \R^n \to \R$,
 inequality constraints $g_i: \R^n \to \R$  and equality constraints $h_j: \R^n \to \R$. 
The \emph{feasible set} is defined as the set of all solutions of the search space $\R^n$ that satisfy the constraints.
We denote by $C_i = \{ x \in \R^n, \mbox{ such that } g_i(x) \leq 0\}$ and $\tilde{C}_j = \{ x \in \R^n, \mbox{ such that } h_j(x) = 0 \} $ the feasible sets associated to the constraints $g_i$ and $h_j$, respectively. The overall feasible set equals
 $  C = \Big(\bigcap\nolimits_{i=1}^p C_i\Big) \cap \left(\bigcap\nolimits_{j=1}^q \tilde{C}_j\right) .
$
%
When $p=0$, by convention $\bigcap_{i=1}^p C_i = \R^n$, and similarly, if $q=0$, $\bigcap_{j=1}^q \tilde{C}_j = \R^n$. 
The feasible set associated to an inequality constraint corresponds to the \emph{sublevel set} of the constraint with value zero where we remind that the sublevel set of a function $g$ with value $\alpha \in \R$ is defined as
$
\Lev^\alpha_{\leq}(g) = \{ x \in \R^n | g(x) \leq \alpha  \}.
$
In a compact way, we will also denote the problem \eqref{eq:MO-constrained} by 
$
(f,g,h) \mbox{ or } (f,C) .
$
When the constraints are convex and lower semicontinuous\footnote{We remind that in a metric space, a function is lower semicontinuous  in $x_0$ if for every sequence $(x_t)_{t \in \mathbb{N}}$ that converges to $x_0$ such that $f(x_t) \to y$, we have $f(x_0) \leq y$.}, the feasible set is convex, as an intersection of convex sets (the sublevel sets of convex functions), and closed as an intersection of closed sets (the sublevel sets of lower semicontinuous functions are closed) as formalized in the next lemma. The proof is immediate and therefore omitted. 
\begin{lemma}\label{lem:convex-feasible}
Assume the constrained problem \eqref{eq:MO-constrained} with $p \geq 0$ inequality constraints that are convex and lower semicontinuous and $q \geq 0$ linear equality constraints. The feasible set $C$  is closed and convex.
\end{lemma}
\hide{
\begin{proof}
The feasible set is the intersection of the sublevel sets $C_i$ which are convex since $g_i$ is convex and closed since $g_i$ is lower semicontinuous. Hence $\left(\bigcap\nolimits_{i=1}^p C_i\right)$ is closed (resp. convex) as intersection of closed (resp. convex) sets. The equality constraints $h_j$ are linear and thus continuous since the search space is finite dimensional. Hence $\tilde{C}_j$ is closed as preimage of a closed set by a continuous function. It is also convex as intersection of the sets $\{x, h_j(x) \leq 0\}$ and $\{x, - h_j(x) \leq 0\}$ which are both convex since $h_j$ and $-h_j$ are linear functions and thus convex. Then $\left(\bigcap\nolimits_{j=1}^q \tilde{C}_j\right)$ is convex and closed as intersection of closed and convex sets. Overall $C$ is closed and convex.
\end{proof}
}

\noindent Given a function $f:\R^n \to \R^m$, a preorder\footnote{A preorder or quasiorder is a binary relation that is reflexive and transitive but not necessarily antisymmetric. We cannot assume the latter here due to the possibility of two solutions $x,y\in\R^n$ being mapped to the same vector $f(x)=f(y)\in\R^m$.} on feasible solutions can be defined by the  \emph{weak dominance relation} $x \preceq_f y$, where 
\begin{equation}
    x \preceq_f y \mbox{ iff } \mbox{for all } i=1,\ldots,m: f_i(x) \leq f_i(y) .
\end{equation}

\noindent We will say $x$ is \emph{better than or equal} to $y$ when $x \preceq_f y$, and \emph{better than} $y$ when $x \preceq_f y$ and $f_j(x) < f_j(y)$ for some $j$. In the latter case, $x$ \emph{dominates} $y$ and we write $x \prec_f y$. 
We simplify the notation from $x \preceq_f y$ to $x \preceq y$ and from $x \prec_f y$ to $x \prec y$ when $f$ is the objective function of the considered optimization problem.

\begin{definition}{\cite[Definition~11.3]{jahn2009a}}
A solution $x\in\R^n$ is called Pareto-optimal, or efficient, for the problem~\eqref{eq:MO-constrained} if it is feasible, i.e., $ x \in C$, and there is no $y \in C$ such that $y \prec x$.
Equivalently, $x$ is Pareto-optimal if it is feasible and for any feasible $y$ with $y \preceq x$, it holds that $f(y)=f(x)$.
%
The set of Pareto-optimal solutions is referred to as the Pareto set and denoted by $\PS^{(f,C)}$. The image under $f$ of solutions from the Pareto set is called the Pareto front.\footnote{This terminology follows the one used in \cite{custodio2011direct} while different terms are used elsewhere, e.g., \cite{sawaragi1985theory}.}
\end{definition}

\noindent We also introduce the notion of weak Pareto optimality.
\begin{definition}{\cite[Definition~11.5]{jahn2009a}}
A solution $x\in\R^n$ is weakly Pareto-op\-ti\-mal, or weakly efficient, for the problem~\eqref{eq:MO-constrained}, if it is feasible and there is no $y \in C$ with $f_i(y) < f_i({ x})$ for all $i=1, \dots, m$.
\end{definition}
\noindent A Pareto-optimal solution is also weakly Pareto-optimal but the reverse is not true in general. When the objectives are strictly convex, then the reverse is true (see Lemma~\ref{lem:stric-convex-WPOisPO}).
The set of Pareto-optimal solutions of the unconstrained problem is referred to as \emph{unconstrained Pareto set} and denoted by $\PS^{(f,\R^n)}$. 
A feasible solution that belongs to the unconstrained Pareto set is Pareto-optimal also for the constrained problem, i.e., $
\PS^{(f,\R^n)} \cap C \subset \PS^{(f,C)}$ (see \cite[Lemma 2.4]{GECCO2025}).

Following the standard terminology, we define the \emph{ideal point} in the objective space as the vector where each coordinate corresponds to the best possible value for the corresponding objective among all Pareto-optimal solutions, while the \emph{nadir point} gives the worst.
\begin{definition}\label{def:ideal-nadir}
The vector $\left(\inf_{x \in \PS^{(f, C)}}f_1(x), \dots, \inf_{x \in \PS^{(f, C)}}f_m(x)\right)$ is called the ideal point of the problem~\eqref{eq:MO-constrained}. 
The vector $\left(\sup_{x \in \PS^{(f, C)}}\!f_1(x),\dots, \sup_{x \in \PS^{(f, C)}}\!f_m(x)\right)$ is called the nadir point of the problem~\eqref{eq:MO-constrained}.
\end{definition}

\subsection{Auxiliary mathematical tools}
We denote by $f(C) + \R^m_+ $~\cite[Page~299]{jahn2009a} the set $\{ y \in \R^m, y_i \geq f_i(x) \mbox{ for some }  x \in C \mbox{ and all } i \in \{ 1, \ldots, m \} \}$.
If all $f_i$ are convex and the set $C$ is convex, then $f(C) + \R^m_+$ is convex. 
\begin{lemma}\label{lem:fconvex}{(\cite[Page~180]{boyd2004convex} and \cite[Lemma 3.1]{GECCO2025} for a proof)} 
Assume that the objective functions $f_i, i=1, \dots, m$, are convex and the set $C$ is a convex subset of $\R^n$. Then the set $
f(C) + \R^m_+ $ is a convex subset of $\R^m$.
\end{lemma}

\noindent We have introduced the notions of Pareto optimality and weak Pareto optimality. While in general not all weakly Pareto-optimal solutions are Pareto-optimal, in the case where the objectives $f_i$ are continuous and strictly convex, weakly Pareto-optimal solutions and Pareto-optimal solutions coincide as stated in the next lemma.
\begin{lemma}{\cite[Lemma~1.3]{attouch2015dynamic}}\label{lem:stric-convex-WPOisPO}
If all objectives $f_i$ are continuous and strictly convex,
 then a solution is Pareto-optimal if and only if it is weakly Pareto-optimal.
\end{lemma}
\noindent We will characterize Pareto-optimal solutions for strictly convex and continuous objectives. By the previous lemma, they are equivalent to weakly Pareto-optimal solutions. We will thus use the following theorem, which allows us to identify weakly Pareto-optimal solutions as solutions of a weighted sum of the objectives, under the assumption that the set $ f(C) + \R^m_+$ is convex.
\begin{theorem}{\cite[Table~11.5, Page~302]{jahn2009a}}\label{Jahn:scalarizing}
\noindent Consider the minimization of a multiobjective problem  $f=(f_1,\ldots,f_m)$ over a non-empty set $C$. Assume $f(C) + \R^m_+$ is a convex set.
A solution $x^\star$ is weakly Pareto-optimal for this multiobjective problem if and only if there exist  $t_1, \ldots, t_m \geq 0$ with at least one $t_i > 0$ such that $
x^\star \in {\rm argmin}_{x \in C} \sum_{i=1}^m t_i f_i(x)$.\footnote{In the same reference~\cite[Table~11.5, Page~302]{jahn2009a} Pareto-optimal solutions are connected to the optima of scalarized functions in the following way: If $x^\star$ is a Pareto-optimal solution of the multiobjective problem, then there exist $t_1, \ldots, t_m \geq 0$ with at least one $t_i > 0$ such that $
x^\star \in {\rm argmin}_{x \in C} \sum_{i=1}^m t_i f_i(x)$. 
Conversely, let $t_1,\ldots,t_m > 0$ be $m$ strictly positive scalars, and let $x^\star \in {\rm argmin}_{x \in C} \sum_{i=1}^m t_i f_i(x)$, then $x^\star$ is Pareto-optimal for the multiobjective problem.
}
\end{theorem}

\noindent One of our proofs relies on the Hilbert projection theorem~\cite[Section~3.1]{hiriart2004fundamentals} that we remind here for the finite dimensional search space $\R^n$ relevant to our case.
\begin{theorem}\label{theo:proj}
Consider a norm $\| . \|$ deriving from an inner product $\langle.,.\rangle$ in $\R^n$ and a non-empty closed convex set $C \subset \R^n$. For every $x \in \R^n$, there exists a unique $y \in C$ such that 
$
\| x - y \| = \inf_{z \in C} \| x - z \| .
$
The vector $y$ is called the projection of $x$ onto $C$ and denoted by $y=\Proj_{C,\|.\|}(x)$.
In addition, $y$ is equivalently characterized by 
$\langle x-y,z-y\rangle \leq 0 $ for all $ z \in C$.
\end{theorem}

\subsubsection*{Union and intersection of feasible sets}
We construct test problems with non-convex feasible sets that are the union of (convex) feasible sets. As seen above, a feasible set associated to a single inequality constraint corresponds to the sublevel set of the constraint with value $0$. 
Hence, the union of feasible sets can be represented by a single constraint given by the minimum of the individual constraint values. This representation is formalized in the next lemma.
\begin{lemma}\label{lem:feasible-mp}
Consider a finite number of inequality constraints $g_i: \R^n \to \R$ for $i=1,\ldots,k$ and the associated feasible sets $C_i$. Then, $C = \bigcup_{i=1}^{k} C_i$ is the feasible set associated to the inequality constraint $g(x) = \min(g_1(x),\ldots,g_k(x))$. If the functions $g_i$ are lower semicontinuous, then $g$ is lower semicontinuous and the set $C$ is closed.
\end{lemma}
\begin{proof}
The first part is a consequence of the following equality between sets:
\maybemath{\{ x | g(x) \leq 0 \} = \{ x | \min(g_1(x), \ldots, g_k(x)) \leq 0 \} = \{ x | \exists i \in \{1,\ldots,k\}, g_i(x) \leq 0 \}.} 
The lower semicontinuity is preserved if we take the minimum of functions, thus $g$ is lower semicontinuous and the sublevel set $\Lev^0_{\leq}(g)$ is closed (see \Cref{lem:convex-feasible}).
\end{proof}
\noindent We will later define functions given by the minimum of some peak functions as \emph{multipeak} functions. The previous lemma hence shows that the feasible set associated with a multipeak function corresponds to the union of the feasible sets associated with each peak function.
Similarly, the intersection of feasible sets corresponding to constraints $g_i$ is equivalent to a single constraint given by the maximum of the constraints $g_i$.
\begin{lemma}\label{lem:maximumconstraints}
Consider $k$ inequality constraints $g_i: \R^n \to \R$ for $i=1,\ldots,k$ and their associated feasible sets $C_i$.
 Then $C = \bigcap_{i=1}^{k} C_i$ is the feasible set associated with $g(x) = \max(g_1(x),\ldots,g_k(x))$. Therefore, there is equivalence between the multiobjective problem with $k$ inequality constraints $(f,(g_1,\ldots,g_k))$ and the problem with a single inequality constraint $(f,\max(g_1,\ldots,g_k))$.
\end{lemma}

\section{Pareto set geometry of constrained multiobjective problems with convex-quadratic or multipeak objectives and constraints}\label{sec:main-theory} 

In \Cref{sec:invariance}, we prove a general invariance property of Pareto sets for constrained multiobjective problems, which allows us to extend the class of problems for which the Pareto set can be characterized. We then mathematically characterize the Pareto set of strictly convex-quadratic problems with a closed convex feasible set in \Cref{sec:sub-31}. In \Cref{subsec:extension_multipeak}, we extend the characterization to multipeak problems where each objective is a multipeak function built from strictly convex-quadratic functions and/or each inequality constraint is a multipeak function. Finally, \Cref{sec:ideal-nadir} shows how to simply compute the ideal and nadir points for such problems.

\subsection{Invariances}
\label{sec:invariance}
We formalize that the Pareto set of a constrained multiobjective problem is invariant to strictly increasing transformations of the objectives, to sign-preserving transformations of the inequality constraints~\cite{oskar2025gecco} and to zero-only-at-zero transformations of the equality constraints.
We define here a sign-preserving transformation as a function $\tau$ such that $\tau(x)>0$ if and only if $x>0$. Equivalently, $\tau$ is sign-preserving if $\tau(x) \leq 0$ if and only if $x \leq 0$.
We also define a zero-only-at-zero transformation $\tau'$ which is such that $\tau'(x) = 0 $ if and only if $x=0$.
 Given an inequality constraint $g$ and its associated feasible set $C = \{ x | g(x) \leq 0 \}$,  $C$ is unchanged if the constraint is instead $\tau(g(x))$ where $\tau$ is sign preserving. Similarly, an inequality constraint is invariant to zero-only-at-zero transformations. We formalize those invariants in a lemma whose proof is immediate. 
 
\begin{lemma}\label{lem:preserve_feasible}
 Consider an inequality constraint $g: \R^n \to \R$. Let $\tau:\Image(g) \to \R$ be a sign-preserving transformation. The feasible set associated to $g$ is equal to the feasible set associated to $\tau \circ g$, i.e., $\{ x | g(x) \leq 0\} = \{ x | \tau( g(x)) \leq 0 \}$
 Similarly consider an equality constraint $h: \R^n \to \R$. Let $\tau':\Image(h) \to \R$ be a zero-only-at-zero transformation. Then the feasible sets associated to the equality constraints $h$ and $\tau' \circ h$ are equal, i.e., $\{ x | h(x) = 0\} = \{ x | \tau'( h(x)) = 0 \}$.
 \end{lemma}

\noindent We now state two equivalences for strictly increasing functions needed later.
\begin{lemma}\label{lem:monotone}
Let $f: \R^n \to \R$ be an objective function and $\varphi: \Image(f) \to \R$  a strictly increasing function. Consider $z_1$ and $z_2 $ in $\R^n$. Then $f(z_1) \leq f(z_2)$ if and only if $\varphi ( f(z_1)) \leq \varphi ( f(z_2))$. Additionally, $f(z_1) < f(z_2)$ if and only if $\varphi ( f(z_1)) < \varphi ( f(z_2))$.
\end{lemma}
\noindent The proof is straightforward and left to the reader. Remark, however, that the strict monotonicity is needed to obtain the equivalence, otherwise only the first implication of the first statement holds.
We can now state the invariance property of a constrained multiobjective problem with regard to strictly increasing transformations of the objectives, sign-preserving transformations of the inequality constraints and  zero-only-at-zero transformations of the equality constraints.
\begin{proposition}\label{prop:transformations_invariant}
Consider a constrained multiobjective problem $(f,g,h)$, where $f=(f_1,\ldots,f_m)$, $g=(g_1,\ldots,g_p)$ and $h=(h_1,\ldots,h_q)$ denote, respectively, vectors of objective functions, inequality constraints and equality constraints.
Consider strictly increasing functions $\{\varphi_i: \Image(f_i) \to \R, i=1,\ldots,m\}$, sign-preserving transformations $\{ \tau_i:\Image(g_i) \to \R, i=1,\ldots, p\}$ and zero-only-at-zero transformations $\{ \tau_i':\Image(g_i') \to \R, i=1,\ldots, q\}$. The original constrained multiobjective problem $(f,g,h)$ and the transformed problem    \begin{equation}\label{eq:transformed}     ((\varphi_1(f_1),\ldots,\varphi_m(f_m)), (\tau_1(g_1),\ldots,\tau_p(g_p)),(\tau'_1(h_1),\ldots,\tau'_q(h_q)))
     \end{equation}
     \noindent have the same Pareto set. In other words, the solution set of a multiobjective problem is invariant to strictly increasing transformations of the objectives, sign-preserving transformations of the inequality constraints and zero-only-at-zero transformations of the equality constraints.
\end{proposition}
\begin{proof}
The feasible set of the problem $(f,g,h)$  equals $C= \left(\bigcap_i \{ x | g_i(x) \leq 0 \}\right) \cap \left(\bigcap_i \{ x | h_i(x) = 0\}\right)$. By \cref{lem:preserve_feasible}, $\{ x | g_i(x) \leq 0 \} = \{ x | \tau_i(g_i(x)) \leq 0 \} $ and $\{ x | h_i(x) = 0 \} = \{ x | \tau_i'(h_i(x)) = 0 \} $ such that the feasible set of \eqref{eq:transformed}, $C'$, equals $C$. Let $x$ be a Pareto-optimal solution of the problem $(f,g,h)$, then $x \in C$ and there is no $y \in C$ such that $y$ is better than $x$ with respect to $\prec_f$.
Since $C=C'$, $x$ is feasible for the problem \eqref{eq:transformed}. Assume that $x$ is not Pareto-optimal for \eqref{eq:transformed}, this means that there exists $y \in C'$ such that $y$ is better than $x$ with respect to $\prec_{\varphi \circ f}$, i.e., $\varphi_i (f_i(y) ) \leq \varphi_i ( f_i(x))$ for all $i=1,\ldots, m$ and $\varphi_j (f_j(y)) < \varphi_j ( f_j(x))$ for some $j$. Yet, since all $\varphi_i$ are strictly increasing, it implies by Lemma~\ref{lem:monotone} that $f_i(y) \leq f_i(x)$ for all $i$ and $f_j(y) < f_j(x)$ for some $j$, which contradicts that $x$ is Pareto-optimal for $(f,g,h)$.
Conversely, let $x \in \PS^{(\varphi\circ f, \tau\circ g, \tau'\circ h)}$.  
Then $x \in C$ and there is no $y \in C$ such that $y \prec_{\varphi \circ f} x$.  
Assume that $x \notin \PS^{(f,g,h)}$.  
Then there exists $y \in C$ such that $y \prec_f x$, i.e.,
$f_i(y) \le f_i(x)$ for all $i$, and $f_j(y) < f_j(x)$ for some $j$.
Since each $\varphi_i$ is strictly increasing, this implies
$\varphi_i(f_i(y)) \le \varphi_i(f_i(x))$ for all $i$ and
$\varphi_j(f_j(y)) < \varphi_j(f_j(x))$ for some $j$
so $y \prec_{\varphi \circ f} x$, a contradiction. Thus,
$
\PS^{(f,g,h)} = \PS^{(\varphi\circ f, \tau\circ g, \tau'\circ h)}.
$
\end{proof}

\noindent This proposition implies that all Pareto set derivations carried out in the sequel for a constrained multiobjective problem $(f,g,h)$ remain valid for the class of transformed problems \eqref{eq:transformed} 
where $\varphi_i$ are strictly increasing, $\tau_i$ are sign-preserving and $\tau'_i$ are zero-only-at-zero. This latter class of problems is considerably more general. Although the Pareto set remains unchanged within the class, problems belonging to it can vary greatly in difficulty for algorithms that are not invariant under such transformations. For example, applying a strictly increasing transformation such as 
$x \mapsto x^\alpha$ for $\alpha \in (0, \frac{1}{2})$, we can transform the convex-quadratic function $f(x) = x^2$ into a non-convex one.
Some illustrations of this effect are discussed later in \Cref{sec:diff-shape}.

\subsection{Convex-quadratic objectives with convex constraints}\label{sec:sub-31}
We consider $m$ strictly convex-quadratic functions $f_i(x) = \frac12 (x - c_i)^\top H_i (x - c_i) + v_i$ where $c_i \neq c_j$ for $i \neq j$, the $m$ symmetric matrices $H_1, \ldots, H_m$ satisfy $H_i \succ 0 $ and $v_1, \ldots, v_m$ belong to $\R$.
We additionally assume $p$ inequality constraints $g_i: \R^n \to \R$ that are convex and lower semicontinuous and $q$ linear equality constraints $h_j: \R^n \to \R$. 
 We include in the above  the possibility to set $p=0$ or $q=0$ (or both) which correspond to the cases of no inequality constraints or no equality constraints (or no constraints at all).
As a particular case, the inequality constraints $g_i$ can be linear. In the sequel, we use $f^\text{c-q}$ to denote the multiobjective problem with the above convex-quadratic objective functions $f = (f_1,\ldots,f_m)$.

\subsubsection{Unconstrained Pareto set geometry}

We first characterize the geometry of the unconstrained Pareto set of the multiobjective problem with convex-quadratic objectives and prove more precisely its mathematical expression. Previous works analyze the case where $m=2$~\cite{toure2019bi,glasmachers2019challenges}.
The proof exploits the property that the Pareto-optimal solutions coincide with the optima of the scalarized function $x \to \sum \theta_i f_i(x)$ for $\theta \in \Sm$ in the unit simplex (see the left plot in Figure~\ref{fig:landscapes-and-weights} for an illustration). Although the result is simple, to the best of our knowledge it has not been stated in this form in the literature; we therefore provide a proof.

\begin{figure}[t]
    \begin{center}
        \includegraphics[height=0.313\linewidth, trim={15pt 15pt 10pt 10pt}, clip]{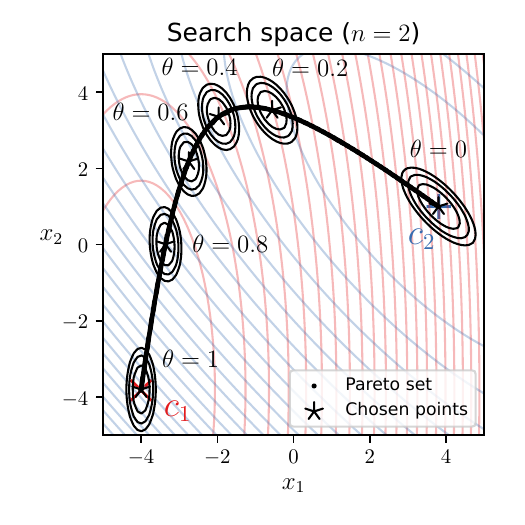}%
        \includegraphics[width=0.687\linewidth, trim={40pt 5pt 40pt 15pt}, clip]{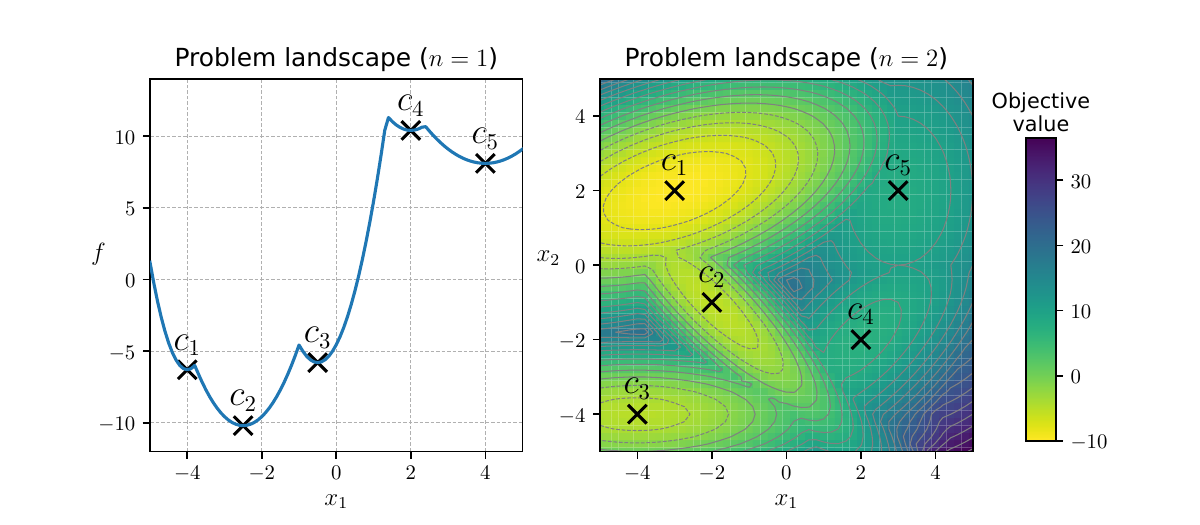}%
    \end{center}
    \caption{
    Left: The level sets of scalarized functions $\ftheta(x) = \theta f_1(x) + (1-\theta) f_2(x)$ for varying $\theta \in [0, 1]$ are shown in black, while the level sets of $f_1$ and $f_2$ are given in red and blue, respectively. The Pareto set (the thick black curve between the objectives' optima) corresponds to the optima of the scalarized functions.
    Uniform sampling of $\theta$ in weight space leads to a non-uniform distribution of Pareto-optimal solutions due to the differing, non-isotropic Hessians of the two objectives.
    Center and right: Multipeak problems with five local minima labeled as $c_i$ in search space with dimension $n=1$ (center) and $n=2$ (right).} 
    \label{fig:landscapes-and-weights}
\end{figure}

\begin{proposition}\label{prop:PS_CQ}
\!Consider an $m$-objective problem with strictly convex-quadratic objectives
$
(f_1(x) = \frac12 (x- c_1)^\top H_1 (x-c_1) + v_1, \ldots, f_m(x) = \frac12 (x - c_m)^\top H_m (x- c_m) + v_m)
$
where the centers $c_i$ are in $\R^n$, $v_i \in \R$ and $H_i \succ 0$. Denote $f^\text{c-q}=(f_1,\ldots,f_m)$. The Pareto set of this unconstrained problem is composed of the optima of the scalarized single-objective functions $\ftheta(x) = \sum \theta_i f_i(x) $ for $\theta \in \Sm$. For each $\theta$, the function $x \mapsto \ftheta(x) $ is also strictly convex-quadratic with positive definite Hessian matrix $H_\theta = \theta_1 H_1 + \ldots + \theta_m H_m$ and unique optimum $c_{\theta} = H_\theta^{-1}( \sum_i \theta_i H_i c_i )$. Moreover,
$
\ftheta(x) = \sum_{i=1}^m \theta_i f_i(x) = \frac12 (x - c_\theta)^\top H_\theta (x - c_\theta) + \sum_{i=1}^m \theta_i v_i \enspace.
$ 
Formally, the Pareto set equals
$
\PS^{(f^\text{c-q},\R^n)} = \left\{ H_\theta^{-1}(\theta_1 H_1 c_1 + \ldots + \theta_m H_m c_m) , \theta \in \Sm \right\}.
$
\end{proposition}
\begin{proof}
 According to \Cref{lem:stric-convex-WPOisPO}, since all objectives are continuous and strictly convex, Pareto optima and weak Pareto optima coincide. In addition, since all objectives $f_i$ are convex, the set $f(\R^n)+\R_+^m$ is convex, by \Cref{lem:fconvex}.  Thus, the assumptions of \Cref{Jahn:scalarizing} are verified. Hence, $x^\star$ is a Pareto-optimal solution if and only if there exist $t_i \geq 0$, $i = 1, \dots, m$, with at least one $t_i > 0$ such that $x^\star \in \argmin_{x\in \R^n} \sum_{i=1}^m t_i f_i(x)$. Since at least one $t_i > 0$, defining $\theta_i = t_i / \sum_j t_j$ yields $\theta \in \Sm$. Since $\argmin$ is unchanged if we minimize $\sum_{i=1}^m t_i f_i(x)$ or $\sum_{i=1}^m t_i/(\sum_j t_j) f_i(x)$, $x^\star$ is a Pareto-optimal solution if and only if $x^\star \in \argmin_{x\in \R^n} \sum_{i=1}^m \theta_i f_i(x)$ for some $\theta \in \Sm$. Let $\theta \in \Sm$ and denote $\ftheta (x) = \sum_{i=1}^m \theta_i f_i(x)$. By \cite[Lemma~3.6]{GECCO2025}, the function $\ftheta(x)$ can be written as $\ftheta(x) = \frac12 (x - c_\theta)^\top H_\theta (x - c_\theta) + \sum_{i=1}^m \theta_i v_i $ with $H_\theta = \theta_1 H_1 + \ldots + \theta_m H_m$ and the unique optimum $c_{\theta} = H_\theta^{-1}( \sum_i \theta_i H_i c_i )$. Overall, we obtain that $x^\star$ is a Pareto-optimal solution if and only if there exists $\theta \in \Sm$ such that $x^\star = H_\theta^{-1}( \sum_i \theta_i H_i c_i )$.
\end{proof}

\noindent When the Hessian matrices $H_i$ are proportional to each other, say, each $H_i = \lambda_i H_0$ with $H_0$ symmetric positive definite and $\lambda_i > 0$, then  $H_\theta = \theta_1 H_1 + \ldots + \theta_m H_m = (\sum_{i=1}^m \theta_i \lambda_i) H_0$ and thus we recover that the Pareto set is the convex hull of the optima $c_i$ of the single-objective functions:
\begin{align}
\PS^{(f^\text{c-q},\R^n)} & = \left\{ \left(\sum_{i=1}^m \theta_i \lambda_i \right)^{-1}  \sum_{i=1}^m \theta_i \lambda_i c_i  , \theta \in \Sm \right\} = \left\{ \sum_{i=1}^m \theta_i' c_i, \theta' \in \Sm \right\}.
\end{align}
In the case of two objectives, the convex hull of the optima corresponds to the segment between the optima.
Figure~\ref{fig:single-2-and-3} illustrates this case (leftmost plot) as well as two additional cases of Pareto sets for arbitrary Hessian matrices with low and slightly increased conditioning (second and third plot from the left) showcasing different Pareto-set geometries in the case of two objectives. The rightmost two plots of Figure~\ref{fig:single-2-and-3} show an example with three objectives where each is a different convex-quadratic function with optima in $c_1$, $c_2$ and $c_3$, respectively.

\begin{figure}[t]
    \centering
    \begin{tabular}{@{}c@{}c@{}c@{}c@{}}
    \includegraphics[width=0.203\textwidth, trim={15pt 15pt 10pt 5pt}, clip]{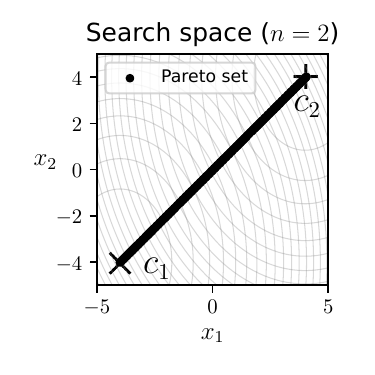}&
    \includegraphics[width=0.203\textwidth, trim={15pt 15pt 10pt 5pt}, clip]{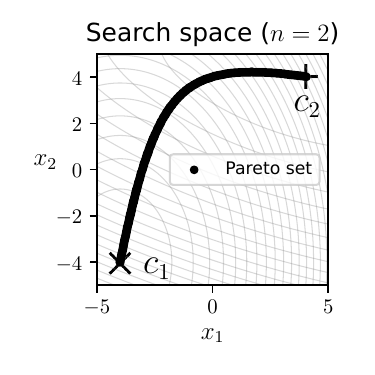}&
    \includegraphics[width=0.203\textwidth, trim={15pt 15pt 10pt 5pt}, clip]{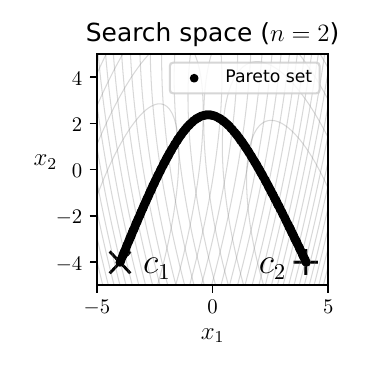}&
    \vspace{5pt}\includegraphics[width=0.391\textwidth, trim={20pt 5pt 10pt 0pt}, clip]{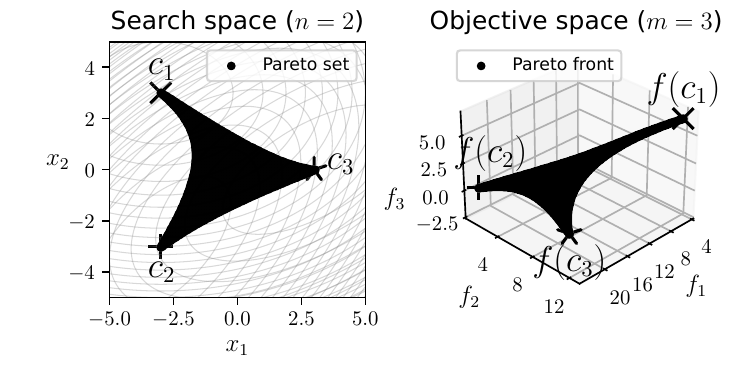}\vspace{-10pt}
    \end{tabular}
    \caption{The leftmost three plots show the Pareto set in the search space for three problems with two convex-quadratic objective functions of two variables and no constraints. The thin gray lines represent the level sets of each objective. The shape and location of the Pareto set depend on the relative positions of the function centers and on the Hessian matrices of the objectives, including both eigenvalues and eigenvector orientations.
    The rightmost two plots show, respectively, the Pareto set and front for a problem with two variables, three convex-quadratic objectives and no constraints.}
    \label{fig:single-2-and-3}
\end{figure}

\subsubsection{Pareto set as the projection of the unconstrained Pareto set on the feasible set}\label{sec:sub-32}

When the objectives $f^\text{c-q}$ of the problem from \Cref{sec:sub-31} are spherical, i.e., each Hessian matrix equals the identity, the Pareto set corresponds to the projection with respect to the Euclidean norm of the unconstrained Pareto set on the non-empty convex feasible set $C$ \cite{GECCO2025}. 
We generalize this result to the case of $m$ strictly convex-quadratic functions. Given $A \succ 0$, and its unique symmetric square root $A^{1/2}$,  we define the inner product $\langle x,y \rangle_A = x^\top A y = (A^{1/2}x)^\top A^{1/2} y$ and its associated norm $\|.\|_{A}$ as
\begin{equation}\label{eq:normA}
\|x\|_{A} = \sqrt{x^\top A x} = \|A^{1/2} x \|_2 .
\end{equation}
We prove that the Pareto set is the projection of the unconstrained Pareto set onto the feasible set. The projection is with respect to a norm that depends on the projected solution and that corresponds to the norm \eqref{eq:normA} with respect to a matrix defined by a convex combination of the Hessian matrices of the convex-quadratic problems. More precisely, the following result, illustrated in Figure~\ref{fig:projection}, holds.
\begin{theorem}\label{theo:mainCQ}
\!Consider an $m$-objective strictly convex-quadratic problem $f^\text{c-q}$ with convex constraints ($g_i: D_i 
\to \R$ are convex and lower-semicontinuous and $h_j$ are linear). We assume that $H_i \succ 0$. Let $C$ be its associated feasible set. Assume that $C$ is non-empty. A solution $x^\star$ is Pareto-optimal if and only if it is the projection of a solution of the unconstrained Pareto set $c_\theta \in \PS^{(f^\text{c-q},\R^n)}$ with respect to the norm $\| . \|_{H_\theta}$ onto the feasible set $C$ where $H_\theta=\theta_1 H_1 + \ldots + \theta_m H_m $ is such that $c_\theta = H_\theta^{-1}(\theta_1 H_1 c_1 + \ldots \theta_m H_m c_m)$ for $\theta \in \Sm$ (see Proposition~\ref{prop:PS_CQ}). In other words,
\maybemath{\PS^{(f^\text{c-q},C)} = \left\{ \Proj_{C,\|.\|_{H_\theta}} (c_\theta), c_\theta =  \left( \sum_{i=1}^m \theta_i H_i \right)^{-1}\sum_{i=1}^m \theta_i H_i c_i , \theta \in \Sm \right\}.}
\end{theorem}
\begin{proof}
From \Cref{lem:convex-feasible}, the feasible set $C$ is closed and convex. In addition, we assume it is non-empty and thus satisfies the assumption needed for the projection \Cref{theo:proj}.
Each objective function $f_i$ is convex (even strictly convex) and thus, according to \Cref{lem:fconvex}, since $C$ is convex, $f(C)+\R^m_+$ is convex. Thus, from \Cref{Jahn:scalarizing}, $x^\star$ is weakly Pareto-optimal if and only if there exist $t_1' \geq 0, \ldots, t_m' \geq 0$ and at least one non-zero $t_i'$ such that $x^\star \in {\rm argmin}_{x \in C} t_1' f_1(x) +  \ldots + t_m' f_m(x)$. Dividing by $\sum_i t_i'$ and calling $t_i = t_i'/\sum_i t_i'$, we have that $t_i \geq 0$ and $\sum_i t_i = 1$ as well as 
$x^\star \in {\rm argmin}_{x \in C} \sum_i t_i f_i(x) = {\rm argmin}_{x \in C} \frac 12 (x - c_t)^\top H_t (x-c_t)   
$
with $c_t = H_t^{-1}( \sum_i t_i H_i c_i ) $ and $H_t = t_1 H_1 + \ldots + t_m H_m$ (see \cite[Lemma~3.6]{GECCO2025}). Using the norm definition \eqref{eq:normA}, we find that 
$
x^\star \in {\rm argmin}_{x \in C} \|x-c_t\|^2_{H_t} = {\rm argmin}_{x \in C} \|c_t-x \|_{H_t}.
$ 
We see that $x^\star$ corresponds to the unique projection of $c_t$ with respect to the norm $\|.\|_{H_t}$ onto the closed non-empty convex set $C$. In addition, from \Cref{prop:PS_CQ} it follows that $c_t$ is Pareto-optimal for the unconstrained problem. Overall, we have shown that $x^\star$ is weakly Pareto-optimal, if and only if it is the unique projection with respect to the norm $\|.\|_{H_t}$ of a solution $c_t = H_t^{-1}( \sum_i t_i H_i c_i ) $ with $t_i \geq 0 $  and $\sum t_i =1$ from the unconstrained Pareto set. Hence, since $f_i$ are strictly convex and the sets of all weakly Pareto-optimal and all Pareto-optimal solutions coincide, $x^\star$ is Pareto-optimal if it is the unique projection of a solution $c_t = H_t^{-1}( \sum_i t_i H_i c_i ) $ with $t_i \geq 0 $  and $\sum t_i =1$ that belongs to the unconstrained Pareto set.
\end{proof}

\begin{figure}[t]
    \centering
    \includegraphics[width=0.6\linewidth,trim={15pt 10pt 5pt 5pt}, clip]{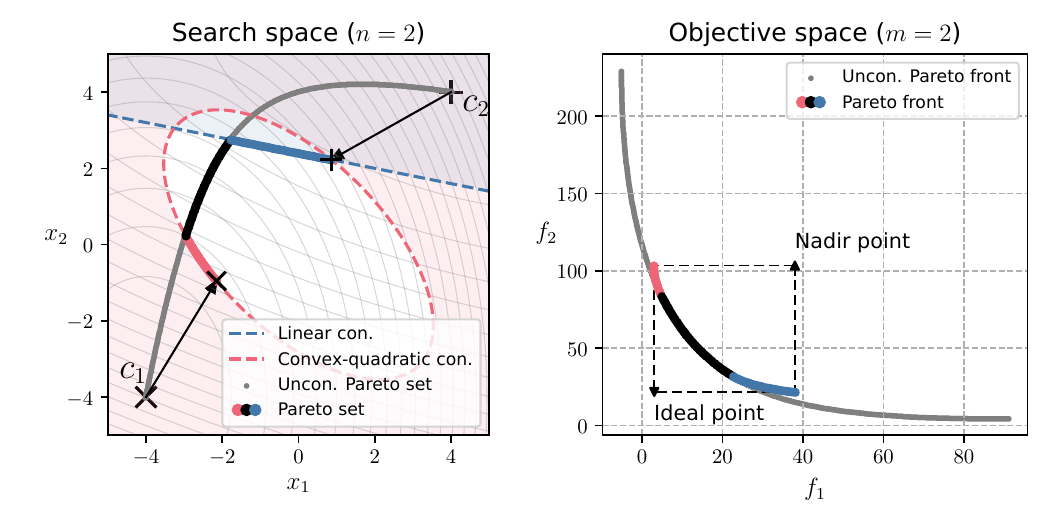}\\[-1em]
    \caption{Illustration of Theorem~\ref{theo:mainCQ}. Left: Search space with unconstrained (gray) and constrained (red, black and blue) Pareto set, together with the objectives' level sets in light gray, one linear constraint (dashed blue) and one ellipsoidal constraint (dashed red). The projections of the objectives' optima, $c_1$ and $c_2$ onto the feasible set are detailed by the two arrows. Right: The corresponding unconstrained (in gray) and constrained (red, black and blue) Pareto front. The colors indicate whether the corresponding Pareto-optimal solutions are projected onto the linear constraint (blue), ellipsoidal constraint (red), or not projected at all (black).
    }
    \label{fig:projection}
\end{figure}

\subsection{Multipeak objectives and constraints}
\label{subsec:extension_multipeak}
In the previous section, we have considered objective functions that are strictly convex and that can model ill-conditioned problems. Another source of difficulty in optimization relates to nonconvexity of the objectives and of the feasible set. We consider here a class of problems from which we can design multimodal problems (i.e., problems that admit more than one local optimum) and non-convex disconnected feasible sets. They are built as the minimum of several functions. More precisely, consider WLG two finite families of functions $\{f_{a_i}: \R^n \to \R , i=1,\ldots,k_a  \} $ and $\{f_{b_j}: \R^n \to \R , j=1,\ldots,k_b  \} $
for $k_a, k_b \in \mathbb{N}$
and the associated multipeak functions
$ f_a(x) = \min_i f_{a_i}(x)$ and $f_b(x) = \min_j f_{b_j}(x)$~\cite{schaepermeier2023multipeaks}.
With a slight abuse of terminology, we refer to this problem as \emph{multipeak}, although it also encompasses unimodal functions (i.e., with a single peak), for instance when $k_a=k_b=1$ and/or  $f_{a_1}$ and $f_{b_1}$ are unimodal. Note also that the term \emph{peak} is more appropriate when maximizing and considering the maximum of several unimodal functions rather than the minimum. We nevertheless retain the term multipeak (instead of multivalley). The center and right plots 
in Figure~\ref{fig:landscapes-and-weights} illustrate multipeak problem instances with five local peaks in dimensions 1 and 2.

\subsubsection{Preliminary results}
To compute the Pareto sets of multiobjective problems with multipeak objectives or constraints, we exploit generic theoretical properties that we first derive in a general setting and later apply to the multipeak case.

Consider a constrained multiobjective problem $(f,C)$ and an arbitrary set $S \subset \R^n$, which may be infinite. We define the operator $S \mapsto \ND^{(f,C)}(S)$, which extracts the set of non-dominated feasible solutions in $S$ for the constrained multiobjective problem $(f,C)$.
Formally, $y \in \ND^{(f,C)}(S)$ if and only if $y \in S \cap C$  and there is no solution in $S \cap C$ that dominates $y$. This set may be empty in the continuous case. Hence, by definition, the set of Pareto-optimal solutions of problem~\eqref{eq:MO-constrained} corresponds to the extraction of non-dominated solutions from the feasible set $C$, i.e.
$
\PS^{(f,C)}=\ND^{(f,C)}(C)
$.
This equation is, however, of limited practical interest as it does not, in general, help to approximate the Pareto set, since $C$ is typically too large to be approximated. Nevertheless, we will explain why it is key to our construction of the Pareto set for multipeak problems.
For a finite set $S$, efficient implementations of the $S \to \ND^{(f,\R^n)}(S)$ operator are available, for example in \href{https://github.com/multi-objective/moocore}{{\ttfamily MOOCore}} or the \href{https://pypi.org/project/moarchiving/}{{\ttfamily moarchiving} module}\footnote{The calculation of all non-dominated solutions within a set of $N$ solutions is possible in time $\mathcal{O}(N\log^{m-1} N)$ for $m\geq 4$ objective functions and in time $\mathcal{O}(N\log N)$ when $m$ is 2 or 3 \cite{kung1975finding}. Computing the non-dominated solutions from the union of two non-dominated sets $A$ and $B$ can itself be done in $\mathcal{O}(|A|\cdot|B|)$ time \cite{karathanasis2025improved}.}. \Cref{prop:ND-incl} presents a key result for computing the Pareto set of the test problems we will consider. It holds under the assumption of the so-called domination property, defined as follows.
\begin{definition}\cite[Definition 2.1]{soriano2026decision}, \cite[Definition~3.2.6]{sawaragi1985theory}
    A multiobjective problem $(f,C)$ is said to satisfy the domination property if for all $x \in C \backslash \PS^{(f,C)}$ there exists $z \in \PS^{(f,C)}$ such that $z \prec x$.
\end{definition}

\noindent The domination property is satisfied  if $f(C)$ is closed and convex and the Pareto set is non-empty (see \cite[Theorem~3.2.12]{sawaragi1985theory} where equivalent conditions to satisfy the domination property are also formulated)
or if $f$ is continuous and $C$ is compact (see \cite[Theorem~2.2]{soriano2026decision}).

\begin{proposition}\label{prop:ND-incl} Consider a (constrained) multiobjective problem $(f,C)$ with its non-empty Pareto set $\PS^{(f,C)}$.
Assume that $(f,C)$ satisfies the domination property.
Let $K \subset C$ be a set such that $\PS^{(f,C)} \subset K$, then
$
\PS^{(f,C)} = \ND^{(f,C)}(K).
$
\end{proposition}
\begin{proof}

We can decompose $K$ into the disjoint union of the set $\ND^{(f,C)}(K)$ and its complementary set (w.r.t.\ $K$) $\ND^{(f,C)}(K)^c$. Let $x \in \ND^{(f,C)}(K)^c$, then $ x$ does not belong to $\PS^{(f,C)}$ since we can find solutions better than $x$ inside $K$ and thus inside $C$. Hence, we have proven that $\PS^{(f,C)} \subset \ND^{(f,C)}(K)$.
Let us show the other inclusion and take $x \in \ND^{(f,C)}(K)$ and prove that it belongs to $\PS^{(f, C)}$. Assume that $x$ does not belong to $\PS^{(f,C)}$ to get a contradiction. 
By assumption, there exists $z \in \PS^{(f,C)}$ such that $z \prec x$. Since $\PS^{(f,C)}$ is a subset of $K$, $z$ belongs to $K$, but then we have found a solution $z$ in $K$ that dominates $x$. This contradicts that $x \in \ND^{(f,C)}(K)$.
\end{proof}
\noindent \Cref{prop:ND-incl} holds for unconstrained problems by taking $C=\R^n$. We infer from the proposition that, if we are able to prove that $\PS^{(f,C)}$ is included in a set $K$ that can be easily and accurately approximated, and such that non-dominated sorting can be performed on this approximation, then we obtain an efficient way to approximate the Pareto set, provided that the domination property is satisfied. This remark is the key observation for computing the Pareto set of problems whose objectives are multipeak combinations of convex quadratic functions as we will see in \Cref{subsec:multipeak}.
To be able to use \Cref{prop:ND-incl} and identify the Pareto set of problems with multipeak objectives based on the pairwise Pareto sets, we need to establish the domination property. While the feasible sets we consider are typically not compact (as they may be unbounded), we are nevertheless able to prove the domination property as if we were in a compact setting using the following proposition.

\begin{proposition}\label{prop:eq_ps}
Let $(f, C)$ be a constrained multiobjective problem with objective function $f$ and feasible set $C$.
Assume that there exists $y \in C$ and $r>0$, such that $y$ dominates all points of $C \setminus B(y,r)$ where $B(y,r)$ denotes the closed ball centered in $y$ with radius $r$.
Then 
$
\PS^{(f, C)} = \PS^{(f, C \cap B(y,r))}.
$
\end{proposition}

\begin{proof}
Let $x \in \PS^{(f, C)}$.  
Then $x$ is not dominated by any point in $C \cap B(y,r)$.  
Moreover, $x \notin C \setminus B(y,r)$, since in that case it would be dominated by $y$.  
Thus $x \in \PS^{(f, C \cap B(y,r))}$. Conversely, let $x \in \PS^{(f, C \cap B(y,r))}$.  
Then $x$ is not dominated by any point in $C \cap B(y,r)$.  
Thus $x$ is also not dominated by $y$ and therefore not by any point in $C \setminus B(y,r)$.  
Thus $x \in \PS^{(f, C)}$.
\end{proof}
\noindent Using \Cref{prop:eq_ps}, we establish the domination property for problems whose multipeak objectives are formed from strictly convex-quadratic functions, which allows us to apply \Cref{prop:ND-incl}.
\begin{proposition}\label{prop:multipeakCQ}
Let $(f = (f_1,\dots,f_m),C)$ be a constrained multiobjective problem with a non-empty closed feasible set $C$. Assume that each objective function is a multipeak function of the form
$
f_k(x) = \min_j \left( (x - c_{k,j})^{\top} H_{k,j} (x - c_{k,j}) + v_{k,j} \right),
$
where each matrix $H_{k,j}$ is symmetric positive definite. Then the problem satisfies the domination property.
\end{proposition}

\begin{proof}
Fix some $y \in C$. Since $H_{k,j}$ is symmetric positive definite, let $\lambda_{k,j} > 0$ denote its smallest eigenvalue. Then for all $x \in \mathbb{R}^n$,
$
(x - c_{k,j})^{\top} H_{k,j} (x - c_{k,j}) \ge \lambda_{k,j}\|x - c_{k,j}\|^2
$
and thus
$
q_{k,j}(x) = (x - c_{k,j})^{\top} H_{k,j} (x - c_{k,j}) + v_{k, j} \ge \lambda_{k,j}\|x - c_{k,j}\|^2 + v_{k,j}.
$

In particular, for each pair $(k,j)$, there exists $r_{k,j} > 0$ such that for all $x$ outside of the ball $B(y,r_{k,j})$, 
$q_{k,j}(x) > f_k(y)$.
Define $r = \max_{k, j} r_{k, j}$. Then, for every $x \in C$ satisfying $\|x - y\| > r$, we have
$
f_k(x) > f_k(y)
 \text{ for all } k=1,\dots,m.
$
Consequently, $y$ dominates all points in $C \setminus B(y,r)$, and therefore $\PS^{(f, C)} = \PS^{(f, C \cap B(y,r))}$ by \Cref{prop:eq_ps}. Since $C \cap B(y,r)$ is closed and bounded, it is compact. Thus, the restricted problem $(f, C \cap B(y,r))$ satisfies the domination property by \cite[Theorem~2.2]{soriano2026decision}. Consequently, for every $x \in C \cap B(y, r)$, either $x \in \PS^{(f, C \cap B(y, r))} = \PS^{(f, C)}$ or there exists $z \in \PS^{(f, C)}$ such that $z \prec x$. Now let $x \in C \setminus B(y,r)$. By construction, $y \prec x$. Since $y \in C \cap B(y,r)$, the domination property on $(f, C \cap B(y, r))$ implies that either $y \in \PS^{(f, C)}$ or there exists $z \in \PS^{(f, C)}$ such that $z \prec y$, and thus $z \prec x$.
\end{proof}

\noindent We now derive another generic result when the feasible set $C$ can be written as the union of sets $C_i$. The Pareto set of the multiobjective problem $(f,C)$ is then included in the union of the Pareto sets of the multiobjective problems $(f,C_i)$.
\begin{proposition}\label{prop:union-feasible}
Consider a constrained multiobjective problem $(f,C)$ where the feasible set $C$ can be written as the union of sets $C_i$, i.e., $C = \bigcup_i C_i$. Then
$
\PS^{(f, \bigcup_i C_i)} \subset \bigcup_i \PS^{(f,C_i)} .
$
\end{proposition}
\begin{proof}
Let $x \in \PS^{(f, \bigcup_i C_i)}$, then by definition $x \in C = \bigcup_i C_i$ and there is no $ y \in C$ such that $f(y)$ is better than $f(x)$. Since $x \in \bigcup_i C_i$, there exists $i_0$ such that $x \in C_{i_0}$. Since there is no $y \in C$ such that $f(y)$ is better than $f(x)$, there is no $y \in C_{i_0}$ such that $f(y)$ is better than $f(x)$. It means that $x \in \PS^{(f,C_{i_0})}$, i.e., $x \in \bigcup_i \PS^{(f,C_i)}$. 
\end{proof}
\noindent This proposition will be used when considering non-convex feasible sets based on the union of sublevel sets of convex-quadratic functions. In this case, we will be able to approximate $\PS^{(f,C_i)}$ for some well-chosen $f$. We will then apply Proposition~\ref{prop:ND-incl} and compute the Pareto set approximation by selecting the non-dominated solutions from the union of approximations of $\PS^{(f,C_i)}$.

\subsubsection{Pareto set of multipeak problems}
We prove in the following proposition that any Pareto-optimal solution of a (constrained) bi-objective multipeak problem $(f_a, f_b)$ belongs to the union of the pairwise Pareto sets $\PS^{((f_{a_i},f_{b_j}),C)}$. This result is already claimed, but not proven in~\cite{schaepermeier2023multipeaks} in the case of an unconstrained problem, i.e., when $C = \R^n$. We prove it in both cases: constrained and unconstrained.
\begin{proposition}\label{prop:PS_multipeaks_constrained}
Let $x^\star$ be a Pareto-optimal solution for the bi-objective multipeak  problem $((f_a, f_b),C)$ where $C$ is a non-empty set ($C=\R^n$ for the unconstrained case). Then
$
x^\star \in \bigcup_{i,j} \PS^{((f_{a_i},f_{b_j}),C)} 
$, that is, $\PS^{((f_a, f_b),C)} \subset \bigcup_{i,j} \PS^{((f_{a_i},f_{b_j}),C)}$.
\end{proposition}
\begin{proof}
Let $x^\star$ be Pareto-optimal for the bi-objective multipeak problem. By definition, $x^\star \in C$. Assume that it does not belong to $\bigcup_{i,j} \PS^{((f_{a_i},f_{b_j}),C)} $, i.e., it does not belong to any $\PS^{((f_{a_i},f_{b_j}),C)}  $.
Using the definition of $f_a$ and $f_b$, let $i_0$ and $j_0$ be the indexes of some active peaks in $x^\star$ (at least one exists per objective), i.e., such that $f_a(x^\star) = f_{a_{i_0}}(x^\star)$ and $f_b(x^\star) = f_{b_{j_0}}(x^\star)$. As assumed above (to get a contradiction), $ x^\star \notin \PS^{((f_{a_{i_0}},f_{b_{j_0}}),C)}  $. Thus, there exists $z \in C$ such that WLG $f_{a_{i_0}}(z) < f_{a_{i_0}}(x^\star)$ and $f_{b_{j_0}}(z) \leq   f_{b_{j_0}}(x^\star)$ (the strict inequality could hold for the second objective instead). However, since $f_a(z) = \min_i f_{a_{i}}(z) \leq f_{a_{i_0}}(z)$, and similarly $f_b(z) \leq f_{b_{j_0}}(z)$, the solution $z$ is such that $f_a(z) < f_a(x^\star) $ and $f_b(z) \leq f_b(x^\star)$, which contradicts that $x^\star$ is Pareto-optimal for $(f_a,f_b)$.
\end{proof}
\noindent Using \Cref{prop:ND-incl}, we can deduce that the Pareto set of the constrained multipeak problem $((f_a,f_b),C)$ corresponds to non-dominated solutions of the set $\bigcup_{i,j} \PS^{((f_{a_i},f_{b_j}),C)}$ as stated in the next corollary.
\begin{corollary}\label{cor:multipeak_nondominated}
Consider a bi-objective multipeak problem $(f_a,f_b)$ and a non-empty feasible set $C$. Assume that the problem satisfies the domination property (for instance, each function can be the minimum of strictly convex-quadratic functions as in \Cref{prop:multipeakCQ}).
Then, 
the Pareto set of the  constrained multipeak problem $((f_a,f_b),C)$ corresponds to the set of non-dominated solutions from the union of pairwise Pareto sets of the peaks, i.e.,
$
\PS^{((f_a, f_b),C)} = \ND^{((f_a, f_b),C)} \left(  
\bigcup_{i,j} \PS^{((f_{a_i},f_{b_j}),C)}
\right).
$
\end{corollary}
\begin{proof}
    Direct application of \Cref{prop:ND-incl} to the result of \Cref{prop:PS_multipeaks_constrained}.
\end{proof}
\noindent This corollary reflects how we later compute the Pareto set approximation of the multipeak problem where each single-peak function is a strictly convex-quadratic function and the feasible set is convex. In this case, for some well-chosen feasible sets, we are able to efficiently compute an approximation of the Pareto set of the pairwise objectives as the projection of the unconstrained Pareto set onto the boundary of the feasible set (see \Cref{theo:mainCQ}). We  then obtain the approximation of the Pareto set of the multipeak problem by finding the non-dominated solutions of the union of the single-peak Pareto set approximations. \Cref{subsec:multipeak} details this approach further and \Cref{sec:visualinspection} gives some concrete examples.

\subsubsection{Pareto set of problems with multipeak-based feasible sets}
\label{subsec:extension_multiconstraints}

So far, we have either considered problems with convex feasible sets or made no assumption on the feasible set. 
A source of difficulty in optimization comes from non-convex disconnected feasible sets. In this section, we discuss the construction of multiobjective problems with such feasible sets.

We have seen in \Cref{prop:union-feasible} that the Pareto set of a multiobjective problem whose feasible set is a union of  sets is included in the union of the Pareto sets associated with each set. We also saw in \Cref{lem:feasible-mp} that a feasible set can equivalently be defined either as the sublevel set of a multipeak function or as the union of the sublevel sets corresponding to the function defining each peak.
Hence, we can use a multipeak function as a constraint to obtain a feasible set given by the union of sets, as summarized in the next corollary. 

\begin{corollary}\label{cor:multipeak_nondominated_constraint}
    Consider a family of multipeak constraints $\{g_{u} = \min (g_{u, 1},\ldots,$ $g_{u, l_u}), u = 1, \dots, d\}$ (i.e., each $g_{u,j}$ is a multipeak function) and define the feasible set associated to the subconstraint $g_{i, j}$ of the multipeak constraint $g_{i}$ as $C_{i, j} = \{ x | g_{i, j}(x) \leq 0 \}$. Then, the feasible set of a problem with these multipeak constraints is equal to 
    $C=\bigcap_{i = 1}^{d} \bigcup_{j = 1}^{l_i} C_{i, j} = \bigcup_{j_1 = 1}^{l_1} \dots \bigcup_{j_d = 1}^{l_d} \bigcap_{i = 1}^d C_{i, j_i}$.
    Additionally, assume a problem $(f,C)$ that satisfies the domination property. Then, the Pareto set of $(f,(g_1, \dots, g_d))$ is equal to
    $
\PS^{(f,g)} = \ND^{(f,g)}\left( \bigcup_{j_1 = 1}^{l_1} \dots \bigcup_{j_d = 1}^{l_d} \PS^{(f,(g_{1, j_1}, \dots, g_{d, j_d}))} \right).
    $
\end{corollary}
\begin{proof}
The proof is a direct consequence of \Cref{lem:feasible-mp} and \Cref{prop:ND-incl}.
\end{proof}

\noindent We can use the above corollary to identify the Pareto set of a problem with objective function $f$ and constraint function $\min g_{a_i}$, by finding the Pareto set of each constrained problem $(f,g_{a_i})$, computing the union of those sets and only keeping the non-dominated solutions from this union.

To compute the Pareto set of each $(f, g_{a_i})$ efficiently, we will choose the objectives $f$ to be either strictly convex-quadratic or multipeak (composed of strictly convex-quadratic) functions and the constraints $g_{a_i}$ to be either linear or convex-quadratic.

\subsection{Computing the ideal and nadir points}\label{sec:ideal-nadir}
In \Cref{def:ideal-nadir}, we define the ideal and nadir points, which are central for certain performance assessment techniques, such as normalization and reference point selection. In the problem setting of \Cref{theo:mainCQ}, they can be easily computed  without resorting to using the approximation of the entire Pareto set. More precisely, when
$\theta \in \Sm$ selects a single objective, i.e., $\theta_i = 1$ and $\theta_j = 0$ for all $j \neq i$, we obtain $H_\theta = H_i$ and $c_\theta = c_i$. In this case, the projection $p_i = \Proj_{C,\|.\|_{H_i}} (c_i)$ of $c_i$ coincides with a minimizer of $f_i$ over $C$.
Since $p_i \in \PS^{(f, C)}$ (by \Cref{theo:mainCQ}), it follows that $p_i$ is a minimizer of $f_i$ over $\PS^{(f, C)}$, i.e., $\min_{x \in \PS^{(f, C)}}f_i(x) = f_i(p_i)$. Therefore, the ideal point of the $m$-objective problem described in \Cref{theo:mainCQ} can be computed using only the solutions $p_1, \dots, p_m$.
See also Figure~\ref{fig:projection} for an illustration.
In the following proposition, we show that even for more general bi-objective problems (not necessarily convex-quadratic), the nadir point can be computed only from the minima of each objective over the Pareto set, if they exist.
\begin{proposition}\label{prop:bi_nadir}
Consider a bi-objective problem and denote by $\PS$ its Pareto set. 
Suppose that there exist $x_1, x_2 \in \PS$ attaining the minimal values of the two objectives, i.e.,
$
f_1(x_1) = \min_{x \in \PS} f_1(x)$, $
f_2(x_2) = \min_{x \in \PS} f_2(x). $
Then the nadir point of the problem defined as $z^{\mathrm{nad}}=(\sup_{x \in 
\PS} f_1(x), \sup_{x \in 
\PS} f_2(x))$ equals
$
\bigl(
f_1(x_2),\,
f_2(x_1)
\bigr).
$
\end{proposition}
\begin{proof}
We show that $f_1(y) \leq f_1(x_2)$ for all $y \in \PS$. Suppose, for a contradiction, that there exists $y \in \PS$ such that $f_1(y) > f_1(x_2)$. Since $y \in \PS$, it cannot be dominated by $x_2$, therefore $f_2(y) < f_2(x_2) = \min_{x \in \PS} f_2(x)$, a contradiction. It follows that $f_1(x_2) = \max_{x \in \PS} f_1(x)$. Similarly $f_2(x_1) = \max_{x \in \PS} f_2(x)$.
The proposition then follows.
\end{proof}
\noindent \Cref{prop:bi_nadir} implies that for the problem described in \Cref{theo:mainCQ} with two objectives, the nadir point can be computed using only $p_1$ and $p_2$. Indeed, as for the ideal point we can use that $f_1(p_1)=\min_{x
\in \PS} f_1(x)$ such that the second coordinate of the nadir point equals $f_2(p_1)$. Similarly, the first coordinate of the nadir point equals $f_1(p_2)$.

\hide{
\begin{proposition}\label{prop:multi_ideal}
Consider a multiobjective problem $(f, C)$, whose Pareto set $\PS$ is \del{a}\new{the} \new{set of} non-dominated \del{set}\new{vectors} of the union of Pareto sets of subproblems whose feasible sets are contained in non-empty feasible set $C$ and whose ideal points exist. We denote \del{with}\new{by} $I$ the set that contains exactly all of these ideal points of subproblems. Furthermore, assume that $(f, C)$ satisfies the domination property. Then the ideal point of $(f, C)$ is
\[
\left(\min_{x \in I} x_1, \dots, \min_{x \in I} x_m\right),
\]
where $x_j$ denotes the $j$-th component of the vector $x$.
\end{proposition}}

Consider now multipeaks problems as in \Cref{cor:multipeak_nondominated} and \Cref{cor:multipeak_nondominated_constraint}. In this case, the ideal point can be computed using the following proposition.

\begin{proposition}\label{prop:multi_ideal}
Consider a multiobjective problem $(f, C)$, whose Pareto set $\PS$ satisfies \maybemath{\PS=\ND^{(f, C)}\left(\bigcup_i \PS^i\right),} where $
\PS^i$ are Pareto sets of subproblems whose feasible set is contained in a non-empty feasible set $C$ and whose ideal point denoted by $I_i (\in 
\R^m)$ exists. Assume that $(f,C)$ satisfies the domination property. Then the ideal point of $(f,C)$ equals $(\min_{z 
\in I} z_1, \ldots,\min_{z 
\in I} z_m)$, where $I=\bigcup_i\{I_i\}$.
\end{proposition}

\begin{proof}
Fix $j \in \{1, \dots, m\}$. We have that $\min_{z \in I} z_j$ is a lower bound for $\{f_j(x) \mid x \in \PS\}$, since every point in $\PS$ also belongs to $\PS^{i}$ for some $i$ and the ideal point of this subproblem, $I_i$, has its $j$-th component greater than or equal to $\min_{z \in I} z_j$.
Furthermore, there exists no greater lower bound, since for every $r > \min_{z \in I} z_j$, at least one subproblem $i$ has $y \in \PS^i$ with $f_j(y)<r$. This is true, since if for some subproblem $i$, we have $f_j(x) \geq r$ for all $x \in \PS^i$, then the $j$-th component of $I_i$ is greater than or equal to $r$, since $r$ is a lower bound for $\{f_j(x)\mid x \in \PS^i\}$. If we would have $f_j(x) \geq r$ for all $x \in \PS^i$ for all subproblems $i$, then all ideal points would have their $j$-th component greater than or equal to $r$. Then we would have, $r \leq \min_{z \in I} z_j$, which would be a contradiction. Since $y \in C$, it follows from the domination property that there exists some $p \in \PS$ such that $f_j(p) \leq f_j(y) < r$, which means that $r$ is not a lower bound for $\{f_j(x) \mid x \in \PS\}$. Therefore, $\inf_{x\in\PS} f_j(x) = \min_{z \in I} z_j$.
\end{proof}

\noindent We have seen that the nadir point of a bi-objective problem can be computed using the Pareto-optimal solutions that minimize the individual objectives, by \Cref{prop:bi_nadir}. Suppose that we have a problem as described in \Cref{prop:multi_ideal} (which includes multipeak problems), where the Pareto sets of our subproblems are $\PS^j$ and the corresponding ideal points are $(f_1(x_{1,j}), f_2(x_{2,j}))$ with $x_{1,j}, x_{2,j} \in \PS^j$ for $j = 1,\dots,d$. Then, the ideal point is $(f_1(x_{1,l}), f_2(x_{2,k}))$ for some $l,k \in \{1,\dots,d\}$ (by \Cref{prop:multi_ideal}) with $x_{1,l}, x_{2,k} \in \PS$. We can indeed ensure that $x_{1,l} \in \PS$, since we can choose $l$ such that $f_1(x_{1,l}) = \min_{j=1,\dots,d}f_1(x_{1, j})$ and $f_2(x_{1,l})$ is the smallest possible among all such choices of $l$. Then, if $x_{1, l}$ is dominated, it is dominated by some $y \in \PS$ (domination property). Such $y$ must satisfy $f_1(y) = f_1(x_{1,l})$ and $f_2(y) < f_2(x_{1,l})$. Since $y \in \PS^j$ for some $j$ (because $\PS=\ND^{(f, C)}(\cup \PS^i)$), we have $y \prec x_{1,j}$, a contradiction since $x_{1,j} \in \PS^j$ by assumption. Similarly, we can choose $x_{2,k} \in \PS$. Thus, we can compute the nadir point of such a bi-objective problem using only the points $x_{1,j}, x_{2,j}$, by \Cref{prop:bi_nadir}.

\section{\COBI: A generator of test problems with analytically known or efficiently computable Pareto sets}\label{sec:testfunctiongenerator}
Using the theoretical analysis of the previous section, we can mathematically characterize the Pareto sets of a wide variety of constrained multiobjective problems. This includes cases with strictly convex or multipeak objectives and arbitrary constraint functions leading to feasible sets that are unions of closed convex sets.
From this foundation, we construct scalable test problems that combine unimodal and multimodal objective functions with linear and convex-quadratic constraints as well as constraints that have disconnected feasible sets  (whose connected components we will call regions).
Moreover, both objective and constraint functions can be parameterized, for example, by varying the location and orientation of their level sets or by applying strictly increasing transformations of the objectives and sign-preserving transformations of the constraints since we have seen in \Cref{sec:invariance} that such transformations preserve the Pareto set. To ease performance assessment, we aim at problem classes with Pareto sets that are numerically and efficiently computable, while still allowing for challenging practically relevant properties such as ill-conditioning, non-separability and multimodality.

In this context, we define below a parameterizable \COBI{} problem generator.  Depending on the chosen parameters, we present in Section~\ref{sec:approximations} methods for efficiently approximating the Pareto sets of the resulting problems with almost arbitrary precision. Section~\ref{sec:visualinspection} visually investigates the \COBI problem properties in more detail.

Note that the \COBI problem generator allows to instantiate bi-objective problems only (hence, its abbreviation for ``Constrained Biobjective Problems''). Although the definitions of the test problems can be generalized to an arbitrary number of objectives, our implementation\footref{foot:cobi} is restricted to two objectives, as we also provide code to numerically approximate the Pareto set efficiently, see \Cref{sec:approximations}.

The most general constrained problem that the test problem generator can produce is the following $n$-dimensional, bi-objective problem with $p$ constraints:\\[-1.5em]
\begin{equation}\label{eq:problem_approx}
\begin{aligned}
 \mbox{minimize } & f(x) = \left(\Phi_1\left(\min_{i \in \{1,\ldots,s_1\}} f_{1, i}(x)\right), \Phi_2\left(\min_{i\in \{1,\ldots,s_2\}} f_{2, i}(x)\right) \right)\\
 \mbox{subject to } & \tau_k\left(g_k(x)\right) \leq 0,\quad k=1,\ldots, p,
\end{aligned}
\end{equation}
where the two objective functions are of the multipeak type, i.e.,
\begin{equation*}
f_{\alpha, i}(x) = \Upsilon_{\alpha, i}\left(\frac12(x - c_{\alpha, i})^\top H_{\alpha, i} (x - c_{\alpha, i}) \right),\quad \alpha \in \{1,2\} \text{ and } i=1, \ldots, s_{\alpha},
\end{equation*}
with 
$c_{\alpha, i} \in \R^n$ the individual optima of convex-quadratic functions, $H_{\alpha, i} \in \R^{n\times n}$ symmetric positive definite matrices and $s_{\alpha}$ the number of peaks in objective $\alpha$. 
The functions $\Phi_\alpha$ and $\Upsilon_{\alpha, i}$ are arbitrary strictly increasing transformations and $\tau_k$ are sign-preserving transformations. As for the inequality constraints $g_k$, $k=1,\ldots, p$, we restrict ourselves to linear functions of the form $g_k^{\text{lin}} = a_k^\top x + b_k$, convex-quadratic functions of the form $g_k^{\text{c-q}} = \frac12(x - c_k)^\top H_k (x - c_k) - d_k$, and their combinations as 
\begin{equation}\label{eq:const-min}
g_k^{\min} = \min \left\{\min_{k_u} \tau_{k_u}\left(g_{k_u}^{\text{lin}}\right), \min_{k_v} \tau_{k_v}\left(g_{k_v}^{\text{c-q}}\right)\right\}, 
\end{equation}
with $\tau_{k_u}$ and $\tau_{k_v}$ being arbitrary sign-preserving transformations. The feasible set of $g_k^{\min}$ equals the union of the feasible sets of the individual constraint functions $g_{k_u}^{\text{lin}}$ and $g_{k_v}^{\text{c-q}}$ (\Cref{lem:feasible-mp}).
In the above, $a_k, c_k \in \R^n$, $b_k \in \R$, $d_k \in \R_+$ and the $H_k \in \R^{n \times n}$ are symmetric positive definite matrices (with $k\in\{1,\ldots,p\}$).
Note that the linear constraints allow to generate box-constrained \COBI problems.

The above used transformations do not alter the Pareto set (see \Cref{prop:transformations_invariant}). However, they allow us to generalize the class of problem difficulties and to model some properties encountered in real-world problems. For example, the sign-preserving transformation $\tau(x)=1$ if $x>0$, and $\tau(x)=0$ if $x\le 0$, 
applied to a constraint retains only information about feasibility and discards all information about how close a solution is to the boundary of the feasible space. This models real-world scenarios in which a solution is evaluated via a simulation that either succeeds or fails without additional feedback about the direction towards feasible solutions.
Optimization problems also often naturally involve non-linear transformations of the objective function(s), for example, to transform multiplicative relationships into linear ones by a strictly increasing transformation
$\Phi(x) = \log_a(x)$, where $a > 1$.
The \COBI test problem generator allows to model such problems as well. Additionally, strictly increasing transformations of the objective functions allow to change the shape of the Pareto front towards highly steep ones, often observed in practice~\cite{kenny2025multi}, as we will see later in Section~\ref{sec:visualinspection}.

\section{Numerical Pareto set approximations} \label{sec:approximations}
In the following, we discuss how to numerically approximate the Pareto sets for the constrained multiobjective problems provided by the \COBI problem generator from Section~\ref{sec:testfunctiongenerator} and give details about our concrete implementation\footref{foot:cobi}. The Pareto set approximations are computed by successively solving the underlying scalarized problems numerically. Since, by \Cref{prop:transformations_invariant}, the Pareto set of a multiobjective problem remains unchanged under strictly increasing transformations $\Phi_\alpha$ of the objective functions and sign-preserving transformations $\tau_k$ of the inequality constraints, we ignore these transformations when approximating the Pareto set.

\subsection{Strictly convex-quadratic objectives}
\label{subsec:singlepeak}
As the base case, we consider the problem defined in \eqref{eq:problem_approx} with $s_1 = s_2 = 1$, i.e., with a single peak per objective function\footnote{To simplify the notation in this section, we drop the second index in $\cdot_{\alpha, i}$.}, and a closed convex feasible set $C$. We can ignore the strictly increasing transformations $\Phi_\alpha, \Upsilon_{\alpha}$ due to the results of Section~\ref{sec:invariance} and thus our objective functions are strictly convex-quadratic. By \Cref{prop:PS_CQ}, the unconstrained Pareto set of this problem coincides with the curve $h(\theta) : [0, 1] \to \R^n$, defined by
\begin{equation*}
    h(\theta) = \left(\theta H_1 + \left(1 - \theta\right) H_2\right)^{-1}\left(\theta H_1 c_1 + \left(1 - \theta\right) H_2 c_2\right).
\end{equation*}

\noindent To generate an approximation of the constrained Pareto set, we project a set of solutions from the curve $h(\theta)$ onto the feasible set $C$. In the simplest case, we can choose the solutions that correspond to equidistant $\theta$ in the weight space $[0, 1]$, i.e., $\theta_i=(i-1)/(N-1)$ for $i\in\{1, \ldots, N\}$, which yields an approximation with $N$ solutions. The quality of such an approximation, with respect to the Euclidean distance between neighboring solutions, heavily depends on the shape of the Pareto set. 
To amend this, we propose to set the weights so that the resulting solutions will be at most $\epsilon$ (for an arbitrary precision $\epsilon > 0$) apart along the unconstrained Pareto set. The procedure to approximate the Pareto set follows these steps:
\begin{enumerate}
    \item Set $\theta_1 = 0$ and, until a solution $h(\theta_r)$ that is at most $\epsilon$ away from $h(1/2)$ is found, search for the solution $h(\theta_{i+1})$ that is at most $\epsilon$ away from $h(\theta_i)$ by bisection in the weight space $[\theta_i, 1/2]$ for $i = 1, \dots, r$.
    \item Repeat the previous step symmetrically, starting at $1$ and using the weight space $[1/2, \theta_i]$. If the last solutions from this and the previous step are more than $\epsilon$ away, find an additional solution that is $\epsilon$ away from $\theta_r$ by bisection in the weight space $[0, \theta_r]$.
    \item For each $\theta_i$, check whether the solution $h(\theta_i)$ is feasible. If it is, add it directly to the Pareto set approximation $\PS_\epsilon$. Otherwise, project it to the feasible space $C$ by solving the constrained problem
    \begin{equation}\label{eq:const-prob}
        \arg\min_{x\in C} (\theta_if_1(x) + (1-\theta_i)f_2(x))
    \end{equation}
    using a numerical solver. If the solution returned by the solver is feasible, i.e., if the sum of all constraint violations is smaller than $10^{-8}$, add it to $\PS_\epsilon$. If the solver fails to return a solution or returns an infeasible solution (within the specified tolerance)\footnote{Note that in the experiments described later in this paper, the solvers have always successfully returned feasible solutions (within the specified tolerance).}, do not change $\PS_\epsilon$ and continue with the next weight $\theta$.
\end{enumerate}

\noindent When the problem from \eqref{eq:const-prob} has only linear constraints, we use the \href{https://pypi.org/project/daqp/}{\texttt{DAQP}}~\cite{arnstrom2022dual} solver (version 0.7.1) from the {\texttt{\href{https://pypi.org/project/qpsolvers/}{qpsolvers}}}~\cite{qpsolvers} module (version 4.8.0) in Python with default parameters.\footnote{The choice of \href{https://pypi.org/project/daqp/}{\texttt{DAQP}} is based on preliminary experiments with a variety of Python-based solvers (namely \href{https://pypi.org/project/cvxopt/}{\texttt{CVXOPT}}, \href{https://pypi.org/project/daqp/}{\texttt{DAQP}}, \href{https://pypi.org/project/piqp/}{\texttt{PIQP}}, \href{https://pypi.org/project/proxsuite/}{\texttt{ProxQP}}, and \href{https://pypi.org/project/quadprog/}{\texttt{quadprog}} (using the Goldfarb/Idnani dual algorithm), all accessed via the {\texttt{\href{https://pypi.org/project/qpsolvers/}{qpsolvers}}} module), in which \href{https://pypi.org/project/daqp/}{\texttt{DAQP}} showed up as the most precise and most CPU-efficient alternative.} When a problem contains also nonlinear constraints, we use the \href{https://pypi.org/project/scs/}{\texttt{SCS}}~\cite{odonoghue:21} solver (version 3.2.7.post2) from the \href{https://pypi.org/project/cvxpy/}{\texttt{CVXPY}}~\cite{diamond2016cvxpy, agrawal2018rewriting} module (version 1.6.5) in Python with parameters
     $\texttt{eps\_abs} = \texttt{eps\_rel} = \texttt{eps\_infeas} = 10^{-12}$,
     $\texttt{max\_iters} = 10^6$.\footnote{The choice of \href{https://pypi.org/project/scs/}{\texttt{SCS}} is also based on preliminary experiments with a variety of Python-based solvers (namely \href{https://pypi.org/project/ecos/}{\texttt{ECOS}}, \href{https://pypi.org/project/scs/}{\texttt{SCS}}, \href{https://pypi.org/project/Mosek/}{\texttt{MOSEK}}, \href{https://pypi.org/project/gurobipy/}{\texttt{GUROBI}}, all accessed via the \href{https://pypi.org/project/cvxpy/}{\texttt{CVXPY}} module), in which no solver performed notably faster or more precisely than \href{https://pypi.org/project/scs/}{\texttt{SCS}}.}
     
This construction does not guarantee a specific size of the resulting Pareto set approximation explicitly. With decreasing $\epsilon$, however, the size of the set $\PS_\epsilon$ will increase further and further---potentially only limited by numerical issues with the solver that computes the projections.

Note that other constructions are possible (for example, those that ensure solutions are at most $\epsilon$ apart in the \emph{objective} space), but these may be more computationally expensive and are beyond the scope of this paper.

\subsection{Multipeak objectives}
\label{subsec:multipeak}
Now suppose that we have a problem defined in \eqref{eq:problem_approx} with arbitrary $s_1, s_2 \geq 1$, i.e., the objective function can have multiple peaks
and the feasible set $C$ is closed and convex
. To numerically compute the Pareto set approximation $\PS_\epsilon$, starting from an empty set, we follow a similar construction as in~\cite{schaepermeier2023multipeaks} for the unconstrained case:
For each pair of peaks $(i, j) \in \{1, \ldots, s_1\} \times \{1, \ldots, s_2\}$, we compute the approximation of the constrained Pareto set $\PS_{\epsilon, i, j}$ for the problem with objective function $\hat f(x) = (f_{1, i}(x), f_{2, j}(x))$ and the same constraints as in our original problem, using the procedure described in Section~\ref{subsec:singlepeak} with our chosen $\epsilon$. Then, for every solution $x \in \PS_{\epsilon, i, j}$, we check whether it is already dominated by some solution $y \in \PS_\epsilon$. If not, then we remove all solutions in $\PS_\epsilon$ that are dominated by $x$ and add $x$ to $\PS_\epsilon$\footnote{The set of non-dominated solutions is retrieved using the \texttt{BiobjectiveNondominatedSortedList} class from the \href{https://pypi.org/project/moarchiving/}{\texttt{moarchiving}} Python module.}.

Since, as above, with decreasing $\epsilon$, the approximation sets for each peak pair become larger and larger (and thus better and better), together with \Cref{cor:multipeak_nondominated}, we expect that also their union, after removing dominated solutions, becomes a better and better approximation of the true Pareto set.

\subsection{Multipeak constraints}
\label{subsec:multiconstraints}
Let us finally discuss how to approximate the Pareto set when the feasible set is not necessarily convex, but the union of convex sets. Suppose that we have a problem as defined in \eqref{eq:problem_approx} with constraints $g_k^{\min}$ of the form \eqref{eq:const-min}. Note that this is a general case that also covers closed convex feasible sets.

Starting from an empty approximation $\PS_\epsilon$, we approximate the Pareto set as follows. We construct a subproblem by replacing every multipeak constraint $g_k^{\min}$ with one of its subconstraints in $\{g_{k_u}^{\text{lin}}\} \cup \{g_{k_v}^{\text{c-q}}\}$. Since such a subproblem contains only linear and convex-quadratic constraints, its Pareto set can be computed as described in Section~\ref{subsec:multipeak}. We generate all such subproblems (corresponding to all possible combinations of selected subconstraints) and compute their Pareto set approximations. For each solution $x$ in some computed Pareto set approximation, we check whether it is dominated by some solution $y \in \PS_\epsilon$. If it is not, then we remove all solutions in $\PS_\epsilon$ that are dominated by $x$ and add $x$ to $\PS_\epsilon$.

\subsection{Even more general problems}
Following the above and in particular Lemmas~\ref{lem:feasible-mp} and \ref{lem:maximumconstraints}, we can generalize our approach further to any (set of) constraints (and objective functions) that are defined as nested minima and maxima of basic linear and convex-quadratic (sub)functions (or in other words, when the feasible set is the intersection and/or union of convex sets). The same holds even for more general problems as long as we know how to compute the Pareto sets for each combination of subfunctions.
Since our implementation of the proposed \COBI problem generator does not cover these cases, we refrain from providing further details.

\section{A visual inspection of the proposed test problems} \label{sec:visualinspection}
In this section, we visually demonstrate the large variety of problem properties encompassed in the \COBI generator, such as ill-conditioning, non-separability and multimodality of both objectives and constraints. 

Depending on the case, we visualize the problem search space (for $n=2$) and/or its objective space (for $m=2$). The search space plots contain:
\begin{itemize}
    \item the level sets of the two underlying objective functions as thin gray lines,
    \item the peaks' optima, labeled with $c_1$ and $c_2$ for unimodal objectives, and with $c_{\alpha, i}$ for the multimodal objectives (where $\alpha\in\{1,2\}$ indicates the objective and $i$ indicates the number of the peak for this objective),
    \item the unconstrained Pareto set, $\PS^{(f,\R^n)}$, in gray,
    \item the Pareto set of the constrained problem, $\PS^{(f,C)}$, in black,
    \item the boundary of the feasible set as dashed, colored lines (red, blue, yellow, \ldots) and the infeasible set in a lighter shade of the same color, i.e., the feasible set is shown in white.
\end{itemize}
On the other hand, the objective space plots contain the unconstrained Pareto front, $f(\PS^{(f,\R^n)})$, in gray and its constrained counterpart, $f(\PS^{(f,C)})$, in black.

The juxtaposition of unconstrained and constrained Pareto sets and fronts provides a geometric validation of the projection theorem. The plots also help to understand the problem landscapes and thus the difficulties faced by solvers, in particular for the disconnected Pareto sets and fronts in the case of multipeak objectives and constraints.

\subsection{Different types of constraints}
As a first step of our inspection, we make sure that all four types of constrained problems from \cite{ma2019evolutionary} can be generated (see the definitions of those types below).
Figure~\ref{fig:types} showcases some examples.  All of them are non-separable, some are ill-conditioned  ((a), (e) and (g)), some have multimodal objectives ((b), (f) and (h)) and some have multimodal constraints ((e) and (h)).

\begin{figure}[t]
    \centering
    \footnotesize
    \begin{tabular}{@{}c@{}c@{}}
    \includegraphics[width=0.5\linewidth, trim={10pt 10pt 0pt 5pt}, clip]{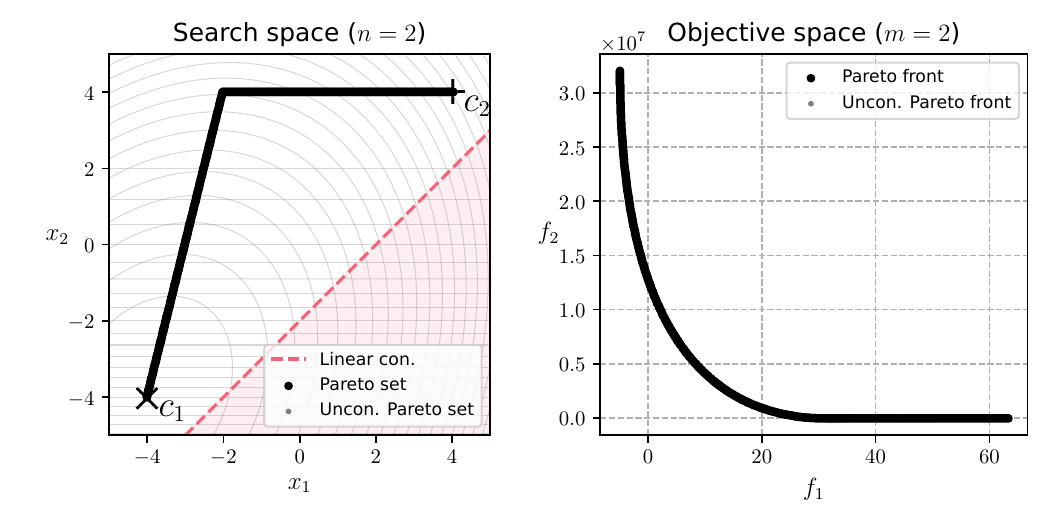}&\includegraphics[width=0.5\linewidth, trim={10pt 10pt 0pt 5pt}, clip]{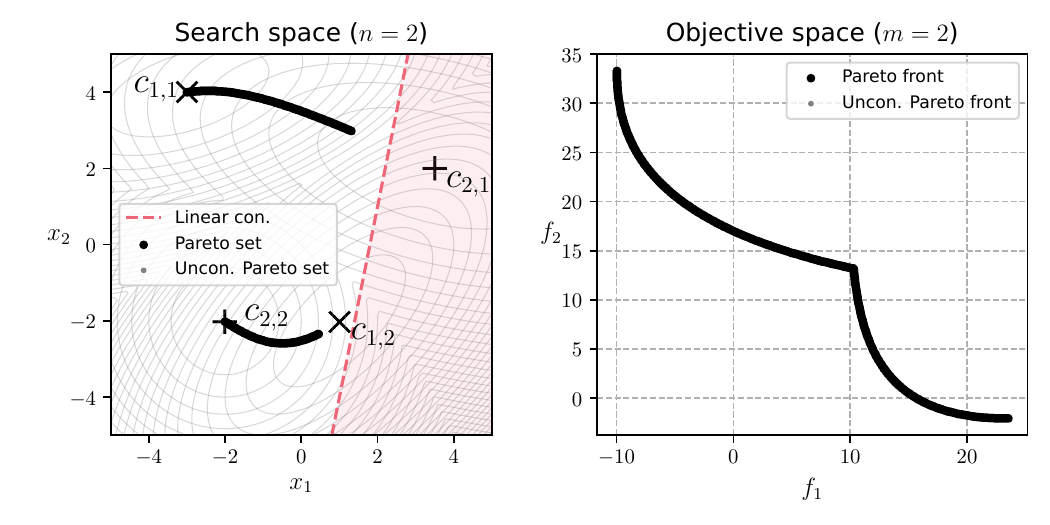}\\
    (a) Type I (unimodal objectives, $f_2$ ill-cond.) & (b) Type I (multimodal objectives)\\[2pt]
    \includegraphics[width=0.5\linewidth, trim={10pt 10pt 0pt 5pt}, clip]{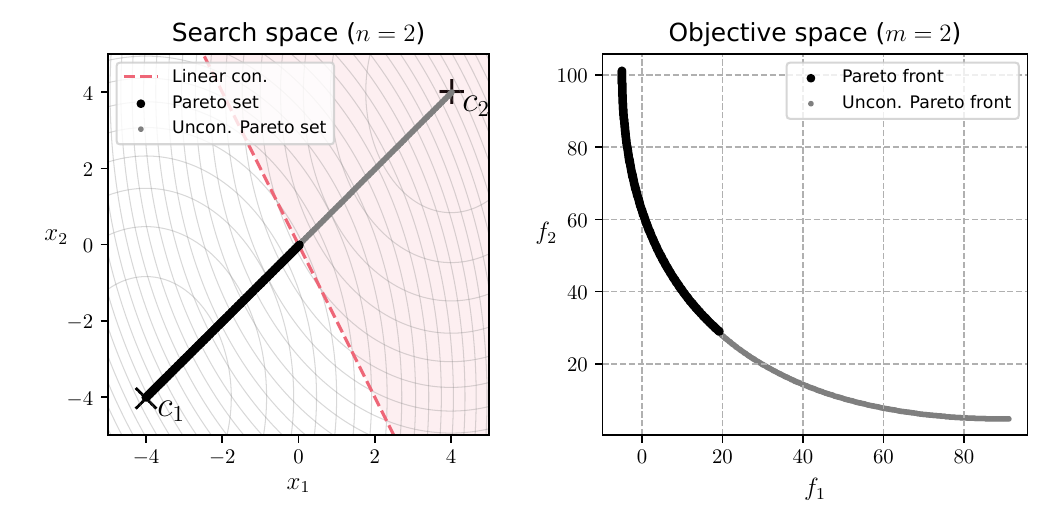} & \includegraphics[width=0.5\linewidth, trim={10pt 10pt 0pt 5pt}, clip]{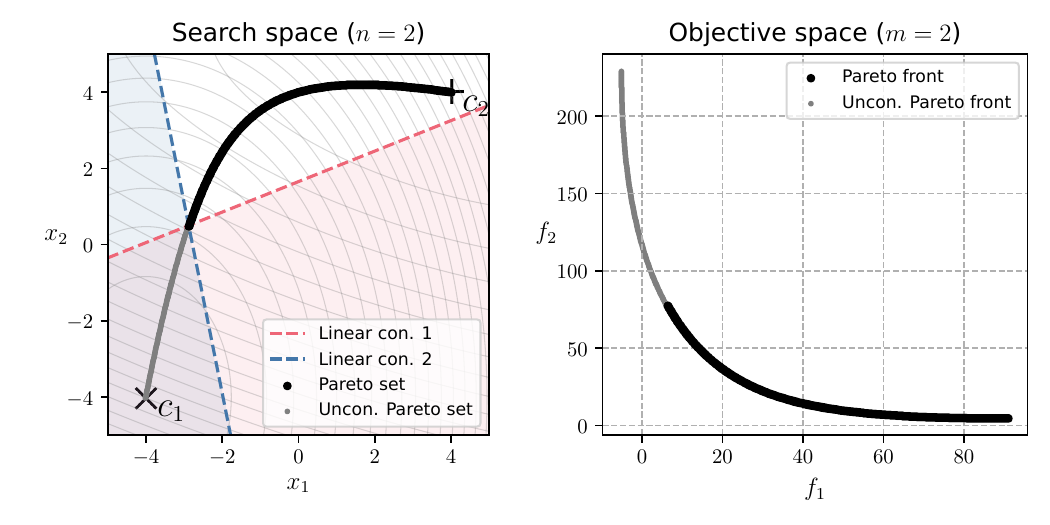}\\
    (c) Type II (single constraint) & (d) Type II  (two constraints)\\[2pt]
    \includegraphics[width=0.5\linewidth, trim={10pt 10pt 0pt 5pt}, clip]{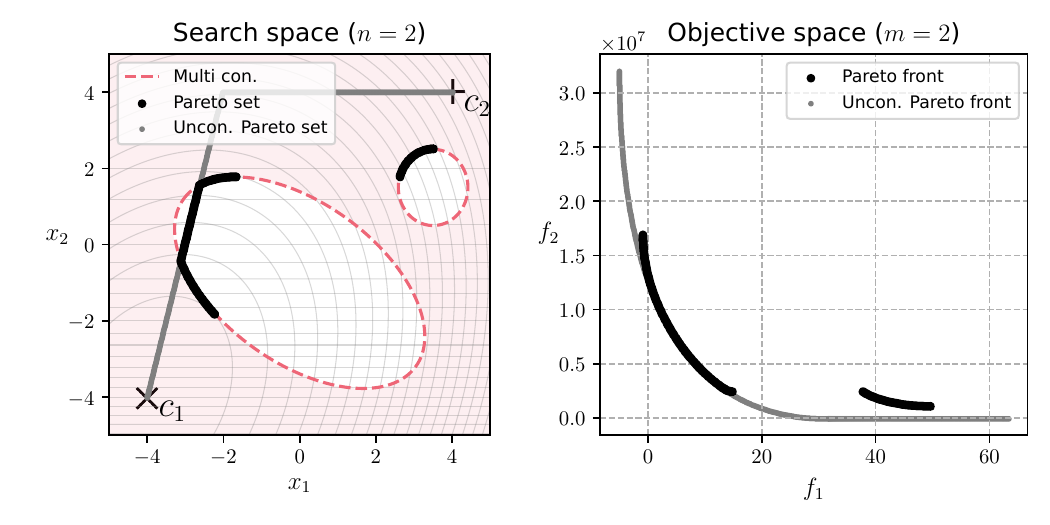} & \includegraphics[width=0.5\linewidth, trim={10pt 10pt 0pt 5pt}, clip]{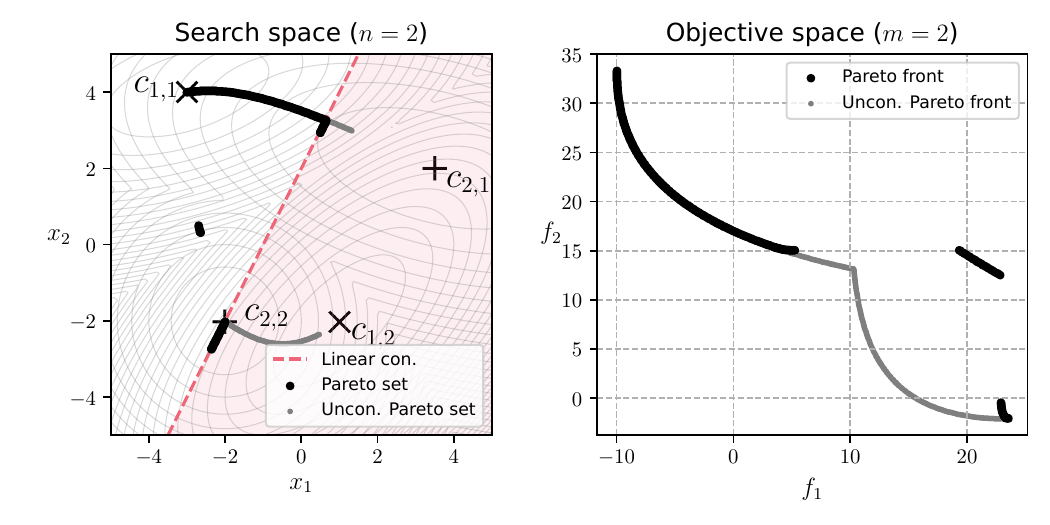}\\
    (e) Type III ($f_2$ ill-cond., multimodal constr.) & (f) Type III (multimodal objectives)\\[2pt]
    \includegraphics[width=0.5\linewidth, trim={10pt 10pt 0pt 5pt}, clip]{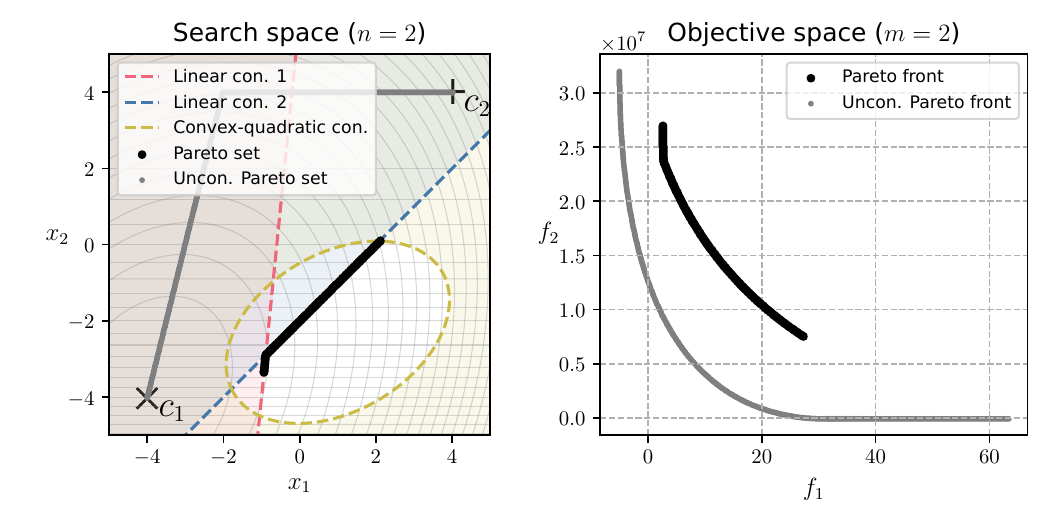} & \includegraphics[width=0.5\linewidth, trim={10pt 10pt 0pt 5pt}, clip]{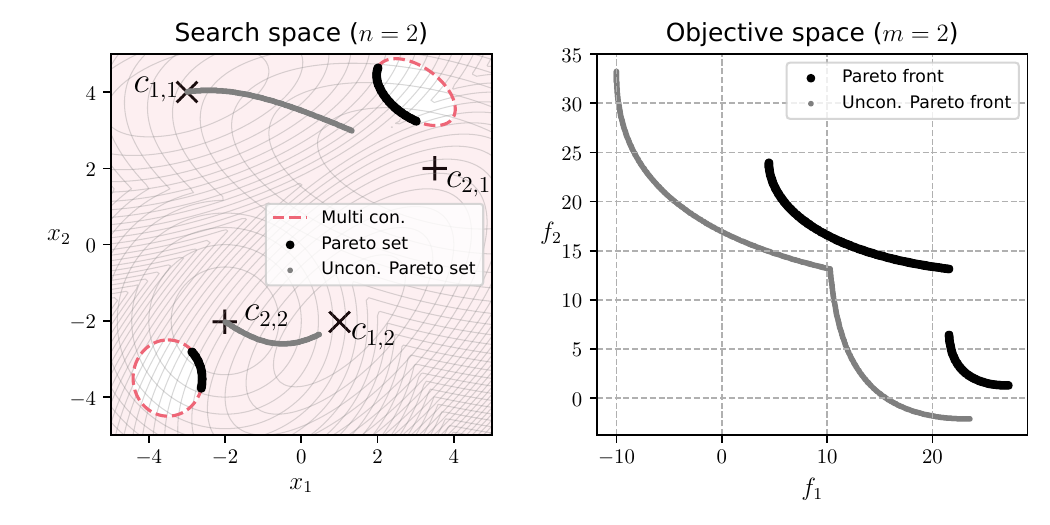}\\
    (g) Type IV ($f_2$ ill-cond., unimodal constr.) & (h) Type IV (multimodal obj.\ and constr.)
    \end{tabular}
    \caption{Four constrained problem types. Type I (first row): The constraints do not change the front. Type II (second row): The Pareto front is a subset of the unconstrained front. Type II (third row): The Pareto front contains part of the unconstrained front plus additional solutions. Type IV (last row): The Pareto front of the constrained and the unconstrained problem are entirely different.}
    \label{fig:types}
\end{figure}

Type I problems are defined as problems where the constraints do not render any unconstrained Pareto-optimal solution infeasible (top row). This is the least interesting case, because no constraint-handling is needed to solve such a problem.\footnote{A multimodal problem could, however, become \emph{easier} to solve if the constraints render locally optimal solutions infeasible.}
Conversely, in Type IV problems, the constraints make all unconstrained Pareto-optimal solutions infeasible (bottom row). In this case, the constrained Pareto front is entirely distinct from the unconstrained one.
In Type II \& III problems, the constraints make only some of the original Pareto-optimal solutions infeasible (middle two rows). According to Theorem~\ref{theo:mainCQ}, the constrained Pareto-optimal solutions are the projections of the unconstrained Pareto-optimal solutions onto the feasible set. As a consequence, we see that Type II  problems (for which a part of the unconstrained Pareto front stays Pareto-optimal while no other Pareto-optimal solution is added) can only occur if all infeasible unconstrained Pareto-optimal solutions are projected back onto the unconstrained Pareto set. This is relatively unlikely, given arbitrary constraints. It can happen, for example, when a linear constraint corresponds to the tangent(s) of the level sets for a given Pareto-optimal solution or when several constraints intersect at a Pareto-optimal solution, see plots (c) and (d) in Figure~\ref{fig:types} respectively. Type III problems, in which the infeasible optimal solutions are projected to new solutions, are more likely to occur than Type II problems. 

\subsection{Different modality of objectives and constraints}
Figure~\ref{fig:multi-multi} shows some additional examples of non-separable problems with multimodal objectives and constraints. 
The first two plots, (a) and (b), show problems with three peaks in the first objective and two in the second. In (a), a single multimodal constraint results in three disconnected Pareto set parts, contained in two of the three feasible regions. When a second convex-quadratic constraint is added, see (b), the Pareto set consists of three parts, each in one of the three (smaller) feasible regions. The third example, (c), shows that multimodal constraints can be used to split the infeasible space into two disconnected regions. Here, the two unconstrained Pareto set parts of the problem with two two-peak objectives are projected into three Pareto set parts. Finally, (d) shows an example of a problem with even higher modality in the objectives (10 peaks in the first objective and 20 in the second) and one multimodal and one linear constraint. This results in seven disconnected feasible regions, of which the smallest four contain Pareto-optimal solutions.

\begin{figure}[t]
    \centering
    \footnotesize
    \begin{tabular}{@{}c@{}c@{}c@{}c@{}}
    \includegraphics[width=0.25\linewidth, trim={20pt 10pt 260pt 5pt}, clip]{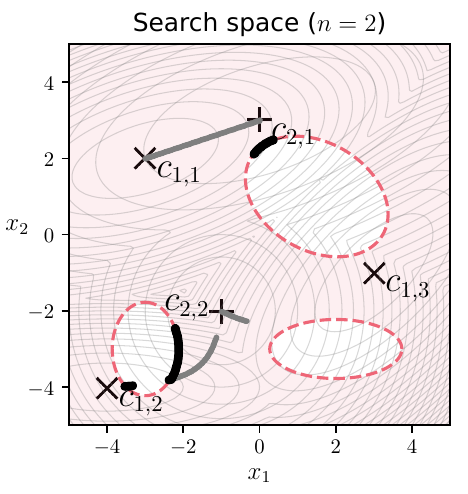}&\includegraphics[width=0.25\linewidth, trim={20pt 10pt 260pt 5pt}, clip]{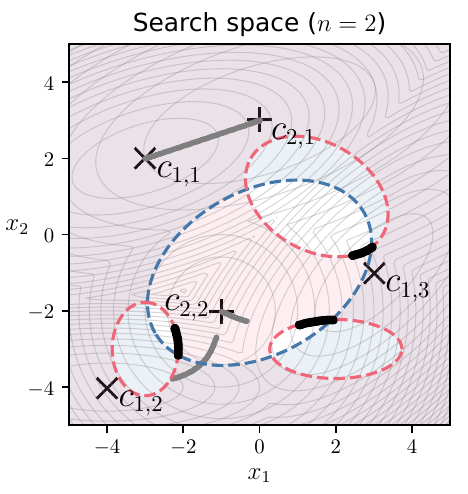}&\includegraphics[width=0.25\linewidth, trim={20pt 10pt 260pt 5pt}, clip]{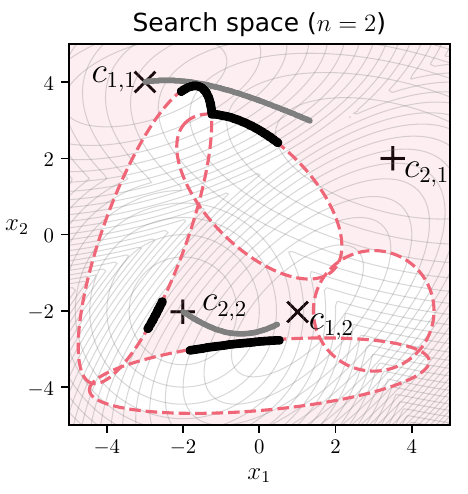}&\includegraphics[width=0.25\linewidth, trim={20pt 10pt 260pt 5pt}, clip]{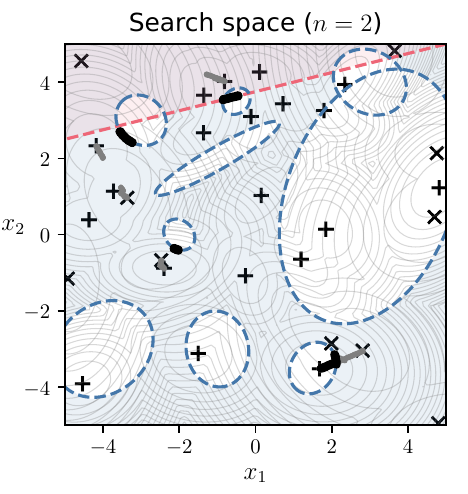}\\
    (a)&(b)&(c)&(d)\\[-10pt]
    \end{tabular}
    \caption{Examples of problems with multimodal objectives and constraints. The legends (and the annotation of peaks' optima in (d)) are removed for better visibility. Colors denote the constraints, while gray and black points show the unconstrained Pareto set and the Pareto set, respectively.}
    \label{fig:multi-multi}
\end{figure}

\subsection{Different Pareto front shapes}\label{sec:diff-shape}
The proposed test function generator also allows to easily create different Pareto front shapes via strictly increasing transformations of objective functions. Figure~\ref{fig:varying-alpha} shows the effect of transformations $\Phi_i^\alpha: \R \rightarrow \R$ with $\Phi_i^\alpha(x) = x^{\alpha_i}$ for varying $\alpha$ applied to the objective functions in the three leftmost plots from Figure~\ref{fig:single-2-and-3}---an effect that has been exploited recently to define so-called ``L-shaped'' Pareto fronts when $\alpha_i=\alpha_j$ for all $i,j$ \cite{kenny2025multi}.

\begin{figure}[t]
    \centering
    \includegraphics[height=0.25\textwidth, trim={65pt 10pt 100pt 5pt}, clip]{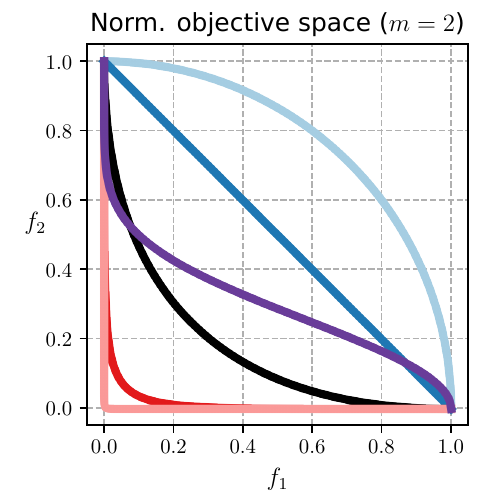}%
    \includegraphics[height=0.25\textwidth, trim={65pt 10pt 100pt 5pt}, clip]{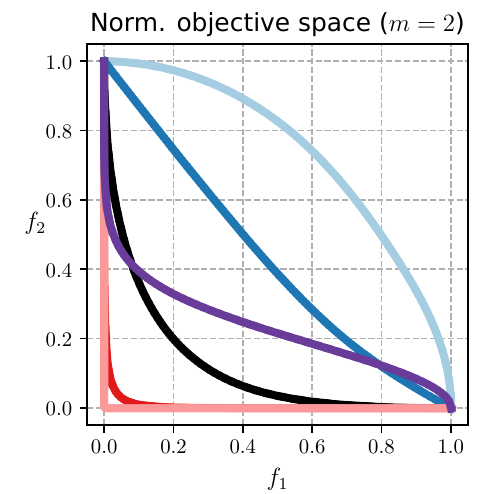}%
    \includegraphics[height=0.25\textwidth, trim={65pt 10pt 100pt 5pt}, clip]{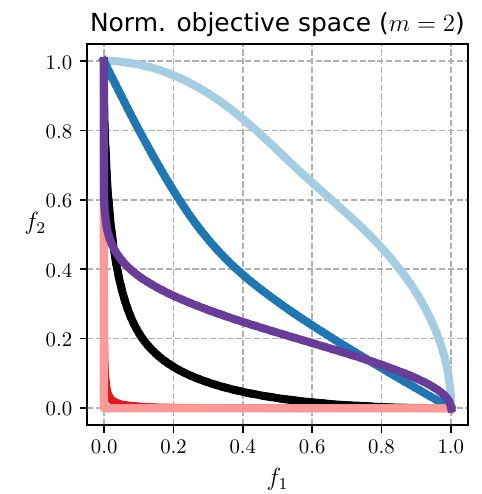}%
    \includegraphics[height=0.25\textwidth, trim={300pt 50pt 5pt 50pt}, clip]{single3-alphas.pdf}
    \caption{Example of some Pareto front shapes achievable with the proposed problem generator (note that the objective space is normalized). To this end, the objective function values $f_1(x)$ and $f_2(x)$ of the three leftmost problems from Figure~\ref{fig:single-2-and-3} are transformed into $\left(f_1(x)\right)^{\alpha_1}$ and $\left( f_2(x)\right)^{\alpha_2}$. Values for $\alpha=(\alpha_1, \alpha_2)$ are encoded in color and given in the legend on the right.}
    \label{fig:varying-alpha}
\end{figure}

We observe that the original convex Pareto front shape (in black) can be transformed into fronts with even stronger curvature (red fronts), but also into concave Pareto fronts (some of the blue fronts) and Pareto fronts which are neither convex nor concave (in violet, e.g., in Figure~\ref{fig:varying-alpha} when $\alpha_1 > 1 > \alpha_2$)---all without changing the original Pareto set. 

While defining test functions with desired Pareto front shapes is not new, the ``classical'' test problems like DTLZ~\cite{deb2005scalable} and WFG~\cite{huband2005scalable} do not use one final transformation on each objective function, but utilize specific shape functions that are applied to a subset of the variables (typically for a single objective only). Note that the transformations, for example in the WFG problems, are not necessarily strictly increasing and thus, their influence on the resulting Pareto front is harder to predict.

\section{Using the \COBI test problem generator for performance evaluation}\label{sec:perfassessment}
The primary purpose of a test function is to support the design  and benchmarking of optimization algorithms. The knowledge of the Pareto set (or a numerical approximation of it with arbitrary precision) for the proposed test functions allows to record and visualize algorithm performance under various scenarios. It is worth mentioning in this context that several performance metrics from the literature naturally depend on the knowledge of a (finite) reference set \cite{ishibuchi2015modified}.

Here, due to space restrictions, we show results only for the impact of problem dimension on a single algorithm in terms of convergence to the Pareto set and its hypervolume indicator value \cite{guerreiro2021hypervolume}. Similarly, other influences on performance can be investigated as well such as when changing the number of constraints, the function modality (i.e., the number of peaks), the convexity/concavity of the problem transformations, the curvature of the constraints etc. As the example algorithm, we use the well-known NSGA-II \cite{deb2002nsgaii} with default setup and a population size of 100 as implemented in the {\ttfamily pymoo} module \cite{pymoo}.
 
Figure~\ref{fig:changing-dimension} shows the effect of search space dimension scaling on NSGA-II for a single-peak problem with one linear and one convex-quadratic constraint. The search space dimension is increased from $n=2$ more or less logarithmically up to $n=40$. To avoid unfair comparisons across dimensions due to a different number of solutions in the Pareto set approximation, we approximated each Pareto set with the numerical procedure from Section~\ref{subsec:singlepeak} using a dimension-dependent precision $\epsilon$, which resulted in $20068 \pm 86$ solutions for all values of $n$. 
%
%
The leftmost plot of Figure~\ref{fig:changing-dimension} shows the hypervolume differences between the Pareto set approximations and the NSGA-II population over time, averaged over 15 independent algorithm runs for each dimension.\footnote{The hypervolume's reference point is chosen as the nadir point.} The other plots of Figure~\ref{fig:changing-dimension} show the projection of all feasible non-dominated solutions found by NSGA-II in the median run to the $x_1$--$x_2$ plane in comparison to the Pareto set approximation in dimensions $n\in\{2, 10, 40\}$. 

We observe that, with increasing dimension, the performance of NSGA-II deteriorates compared to the hypervolume of the Pareto set approximations.
In the search space projection plots, we also see that NSGA-II has difficulty finding the extremes of the Pareto set---an observation that cannot be made from hypervolume values alone.

\begin{figure}[t]
    \centering
    \includegraphics[width=0.25\linewidth, trim={10pt 5pt 7pt 5pt}, clip]{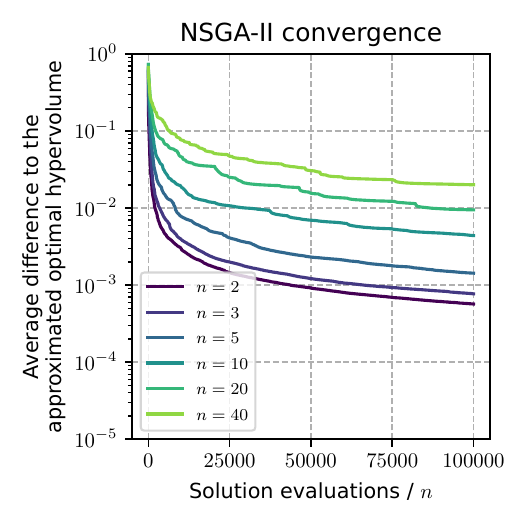}\includegraphics[width=0.25\linewidth, trim={20pt 10pt 255pt 5pt}, clip]{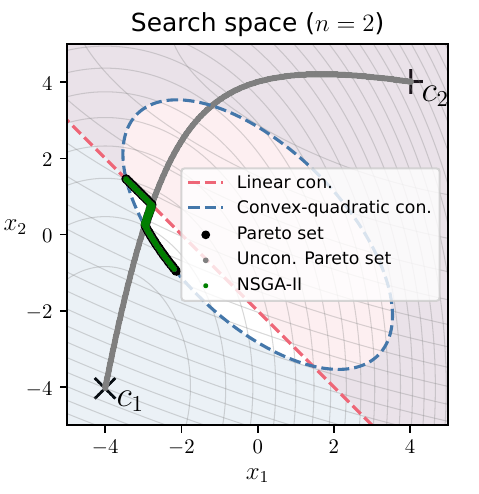}\includegraphics[width=0.25\linewidth, trim={20pt 10pt 255pt 5pt}, clip]{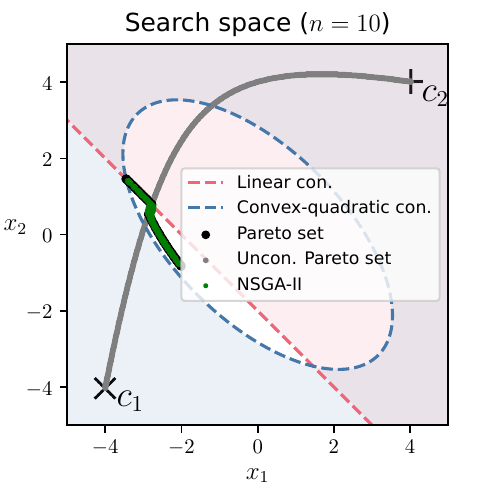}\includegraphics[width=0.25\linewidth, trim={20pt 10pt 255pt 5pt}, clip]{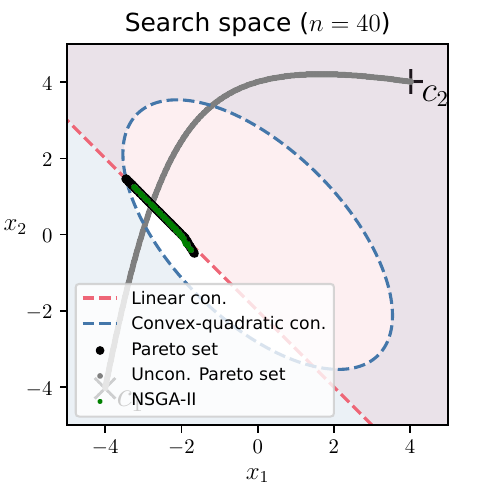}\\[-0.5em]
    \caption{The effect of search space dimension scaling on NSGA-II performance on a single-peak problem with one linear and one convex-quadratic constraint. The plot on the left presents the difference between the hypervolume achieved by NSGA-II (averaged over 15 algorithm runs) and the approximated optimal hypervolume. 
    The other three plots show, from left to right, the projections to the $x_1$--$x_2$ plane of the Pareto set approximation and all non-dominated solutions found by the median run of NSGA-II for search space dimensions $n=2, 10$ and $40$, respectively. They visualize the same effect in the problem search space.}
    \label{fig:changing-dimension}
\end{figure}


\bibliographystyle{plain}
\bibliography{bibtex}

\end{document}